\documentclass[a4paper,10pt]{article}

\usepackage[top=4cm, bottom=3cm, left=3.2cm, right=3.2cm]{geometry}
\usepackage{amsmath}

\usepackage{amsmath}
\usepackage{amsthm}
\usepackage{amssymb}
\usepackage{amstext}
\usepackage{dsfont}
\usepackage{graphicx}
\usepackage{subfig}
\usepackage[dvips]{epsfig}
\usepackage{texdraw}
\usepackage{color}
\usepackage{url}
\usepackage{a4wide}
\usepackage{paralist}
\usepackage{booktabs}
\usepackage{nicefrac}
\usepackage{float}

\newtheorem{thm}{Theorem}

\newtheorem{lem}{Lemma}

\theoremstyle{definition}

\theoremstyle{remark}
\newtheorem{rem}{Remark}

\newcommand{\dint}{\displaystyle\int}
\newcommand{\dsup}{\displaystyle\sup}

\newcommand{\R}{\mathbb{R}}
\newcommand{\E}{\mathbb{E}}

\newcommand{\J}{\mathcal{J}}

\numberwithin{equation}{section} \numberwithin{lem}{section}
\numberwithin{thm}{section} \numberwithin{prop}{section}
\numberwithin{cor}{section} \numberwithin{rem}{section}

\begin{document}

\title{Modeling of a diffusion with aggregation: rigorous derivation and numerical simulation}

\author{Li Chen\footnotemark[1], \; Simone G\"ottlich\footnotemark[1], \; Stephan Knapp\footnotemark[1]}

\footnotetext[1]{University of Mannheim, Department of Mathematics, 68131 Mannheim, Germany (chen@math.uni-mannheim.de, goettlich@uni-mannheim.de, stknapp@mail.uni-mannheim.de).}

\date{\today}

\maketitle

\begin{abstract}
\noindent
In this paper, a diffusion-aggregation equation with delta potential is introduced. Based on the global existence and uniform estimates
of solutions to the diffusion-aggregation equation, we also provide the rigorous derivation from a stochastic particle system while 
introducing an intermediate particle system with smooth interaction potential.
The theoretical results are compared to numerical simulations relying on suitable discretization
schemes for the microscopic and macroscopic level. In particular, the regime switch where the analytic theory fails is numerically analyzed 
very carefully and allows for a better understanding of the equation.
\end{abstract}


\maketitle

\section{Introduction}
\label{sec:Intro}

In the last decades, diffusion-aggregation equations of the following type
\begin{eqnarray*}
\partial_t u-\nabla\cdot (a\nabla u-u\nabla V* u)=0, \quad x\in \R^d
\end{eqnarray*}
have been widely studied in the literature. One prominent example is the so-called Keller-Segel system which corresponds to the case that $V(x)=C(d)/|x|^{d-2}$ is the fundamental solution of the Poisson equation. It is well-known that depending on the choice of the initial datum, the solution to the Keller-Segel system 
might exist globally and blow-up in finite time, see for example \cite{DP,JL92},
or \cite{Bed11} for more general potentials of the form $V(x)=1/|x|^\gamma, \gamma<d-2$.

The aggregation phenomena can be motivated by several effects such as flocking and swarming in biological systems \cite{BV2006,BCM07,TBL06}
or interacting granular media \cite{BCP97,CMV03,T00}.  
Moreover, in \cite{L2007}, it has been shown that the local and global existence of the solution to the aggregation equation, i.e. $a=0$, depends on the regularity of $V$. For instance, for the potential $V(x)=e^{-|x|}$ only local existence can be proved while for $V(x)=e^{-|x|^2}$ the global existence holds.
It is also known that in the case of a power-like potential $V(x)=|x|^\alpha$, $2-d\leq \alpha<2$, the smooth solution of the aggregation equation generates finite time blow-ups, see \cite{BB2010,CDFLS11,DHongjie2011,JV2013}.

In this paper, we focus on the case that the aggregation potential is a delta distribution. More precisely, the problem is reduced to the following diffusion-aggregation equation:
\begin{eqnarray}
\label{eq:macro}
\partial_t u - \nabla \cdot (a \nabla u-2bu \nabla u) = \partial_tu-\Delta ((a-bu)u)=0,
\end{eqnarray}
where $a$ and $b$ are both positive constants and the initial data is given by a non-negative density $u(x,0)=u_0(x)\geq 0$.
The problem can be obtained as a mean-field limit of the following interaction stochastic particle system:
\begin{eqnarray}
\label{eq:micro}
&&d X^i(t)= \sqrt{2a}\,dB^i(t) +\frac{1}{N}\sum_{j\neq i} \nabla V_\varepsilon(|X^i(t)-X^j(t)|)dt,\\
&&X^i(0)=\xi_i, \quad \mbox{ i.i.d. random variables with probability density function (pdf) } u_0 \nonumber
\end{eqnarray}
with $B^i$ being independent standard Brownian motions for each particle $i$.
Furthermore, the potential $V_\varepsilon(r)=\varepsilon^{-d}V(r/\varepsilon)$ with $\dint_{\R^d}V(x)dx=2b$ is considered.
The large particle limit $N\rightarrow \infty$ can be understood in the following sense.
For fixed $\varepsilon>0$, the particle model converges heuristically after applying It\^{o}'s formula to an intermediate non-local problem 
for $u^\varepsilon$, i.e.\ the distribution of the random variable $X^i_t$ at time $t$: 
\begin{align}
\partial_tu^\varepsilon -\nabla\cdot (a\nabla u^\varepsilon -u^\varepsilon\nabla V_\varepsilon * u^\varepsilon)=0 \text{ with } u^\varepsilon(x,0) = u_0(x). 
\label{eq:interm_model}
\end{align}
As $\varepsilon\rightarrow 0$, $V_\varepsilon\rightarrow 2b\delta$, we get that the limit $u$ of $u^\varepsilon$ satisfies the following diffusion-aggregation equation:
$$
\partial _t u-a\Delta u+b\Delta u^2=0.
$$
This equation equipped with logistic reaction has been studied in \cite{CDL17} on a bounded domain with different boundary conditions. Both existence and finite time blow-up results have been obtained there. Using the transformation $v=u-\frac{a}{2b}$, the equation can be rewritten as a backward porous media equation 
\begin{eqnarray*}
&&\partial_tv+b\Delta v^2=0
\end{eqnarray*}
which has a special solution (backward Barenblatt solution) in $d$ dimensions
\begin{eqnarray*}
v(t,x)=\frac{1}{b(T-t)}\Big((T-t)^\frac{2}{d+2}-\frac{|x|^2}{4(d+2)}\Big)_+.
\end{eqnarray*}
We note that the constant $\frac{a}{2b}$ plays a crucial role in terms of a
threshold to get global existence or finite time blow-up. 
In particular, we prove that for $0\leq u_0< \frac{a}{2b}$ and $\int u_0(x)dx<\infty$, the weak solution exists globally. 
Under further assumptions, we get that the solution is smooth and uniform estimates for the solution to the intermediate problem \eqref{eq:interm_model} hold,
see section \ref{sec:solva}. 
These results are then used to prove the rigorous convergence from many stochastic particle system to the trajectory of the diffusion-aggregation problem \eqref{eq:macro}, see section \ref{sec:RigorousMeanField}.
In the literature, a variety of similar results can be found for such convergence proofs.
The rigorous mean-field limit and the propagation of chaos with bounded Lipschitz potential has been introduced in 1991, see \cite{S1991}. 
More than 20 years later, the case with potential $V(x)=|x|^\alpha$, $\alpha\in (-1,0)$ has been proven, see \cite{GQ15}. 
The microscopic derivation of a two-dimensional Keller-Segel system is recently given in \cite{GP2017} while the derivation of the multi-dimentional system 
starting from different cut-off interaction particle systems is done in \cite{HL2017,LY2017}. 
Parallelly, the derivation of porous medium equations with exponent 2 from large interacting particles systems
has been introduced in 1990, see \cite{Oel90}. These results have been later improved in \cite{JM98,Phi07}. 
Since the aggregation effect we consider has the backward porous medium structure, we basically follow the idea taken from the derivation of porous medium equations.
However, we derive estimates according to the well-posedness of the diffusion-aggregation equation.

For our numerical investigations,
we impose a problem-adapted numerical scheme 
to better illustrate the transition from the diffusion to the aggregation regime of the equation \eqref{eq:macro}. 
We prove 
that the numerical method is positive preserving independent on the coefficients $a,b$ which is the main difference to a classical finite difference approximation,
see section \ref{sec:disc}. 
In the derivation of the numerical scheme we use ideas for degenerated parabolic equations \cite{Bessemoulin2012,Buerger2008,Liu2011,Pandian1989} as well as techniques used for chemotaxis models, see e.g.\ \cite{ChertockKurganov2008, ChertockKurganovWang2012}.
We numerically study the order of convergence 
and validate the scheme by examining the difference to the microscopic model, see section \ref{sec:num_res}.
To study the convergence of the microscopic model to the macroscopic equation, we introduce an efficient way to generate sample paths of the 
stochastic particle system \eqref{eq:micro}.
Since we use a superposition of Barenblatt profiles as initial densities, the computation of the pseudo-inverse and the use of the inverse transformation method \cite{KloedenPlaten1992} leads to an exact and efficient way to generate the initial random numbers for the approximation of the particle system.

\section{Solvability and uniform estimates}
\label{sec:solva}
This section is devoted to the solvability of the intermediate and limiting (macroscopic)
diffusion-aggregation problem.
Therefore, the section is divided into three parts: We first show the global existence and uniqueness of the 
non-local intermediate problem and the corresponding uniform estimates in $\varepsilon$. 
Then, the global solvability of the diffusion-aggregation problem is obtained by taking the limit $\varepsilon\rightarrow 0$. 
Finally, the error estimates for small $\varepsilon$ are given. These three results are the main ingredients for the mean field limit in section \ref{sec:RigorousMeanField}.

\subsection{Solvability of the intermediate problem}
As already noted in the introduction, the intermediate problem for $u^\varepsilon$ is
\begin{eqnarray}\label{probinter} 
\partial_tu^\varepsilon-\nabla\cdot(a\nabla u^\varepsilon- u^\varepsilon\nabla V_\varepsilon*u^\varepsilon)=0,&& x\in \R^d , \\
\nonumber u^\varepsilon(x,0)=u_0(x),&&
\end{eqnarray}
where $\dint_{\R^d}V_\varepsilon(x) dx=2b$.
From \cite{Maj84}, we know the
following standard estimates that are frequently used in our proof. For any multi-index $\alpha$ with $|\alpha|=s>\frac{d}{2}+1$, it holds for  $f,g\in H^s(\R^d)$ that
\begin{eqnarray}\label{ineq1}
\|D^\alpha (fg)\|_{L^2}\leq C(\|f\|_{L^\infty}\|D^sg\|_{L^2} +\|g\|_{L^\infty}\|D^sf\|_{L^2}),\\ [1ex]
\|[D^\alpha, f]g\|_{L^2}\leq C(\|Df\|_{L^\infty}\|D^{s-1}g\|_{L^2}+ \|g\|_{L^\infty}\|D^sf\|_{L^2}),\label{ineq2}
\end{eqnarray}
where $C$ depends on $d$ and $s$.

\begin{thm}\label{thmexisV}
Suppose that $0\leq u_0\in H^s(\R^d)\cap L^1(\R^d)$ ($s>\frac{d}{2}+1$) and $\|u_0\|_\infty< \frac{a}{2b}$, then problem \eqref{probinter} has a unique solution $u^\varepsilon\in L^\infty(0, \infty; H^s(\R^d)\cap L^1(\R^d))\cap L^2(0, \infty;H^{s+1}(\R^d))$ and $0\leq u^\varepsilon (x,t)<\frac{a}{2b}$ with the following estimates
\begin{eqnarray}
\|u^\varepsilon(\cdot,t)\|_{L^1(\R^d)}= \|u_0\|_{L^1(\R^d)},\quad
 \|u^\varepsilon\|_{L^\infty(0,\infty; L^2(\R^d))}+\|Du^\varepsilon\|_{L^2(0,\infty;L^2(\R^d))}\leq C,
\end{eqnarray}
where $C$ is a constant independent of $\varepsilon$.
\end{thm}

\begin{proof}
We use the standard Banach fixed-point theorem to prove the local existence of solutions.
Furthermore, we give additional estimates for any fixed $\varepsilon$, so that the existence
can be extended to arbitrary long times.
In the end of the proof, we present the uniform in $\varepsilon$ estimates of the solution.

{\bf Step 1:} Local existence of solution for any fixed $\varepsilon$.
Without loss of generality, we assume $0\leq u_0(x)\leq \frac{a}{2b}-\eta$ with $0<\eta\ll 1$. Let
\begin{eqnarray*}
\mathcal X&=&\Big\{u\in L^\infty(0,T^*;H^s(\R^d)\cap L^1(\R^d)): u(x,0)=u_0(x), \\
&&\hspace{1cm} \|u\|_{L^\infty(0,T^*;H^s(\R^d))}\leq 2\|u_0\|_{H^s(\R^d)}:=\tilde M
,\,0\leq u(x,t)\leq \frac{a}{2b}-\eta:=M,
\\
&&\hspace{1cm} \|u(\cdot,t)\|_{L^1(\R^d)}=\|u(\cdot,t)\|_{L^1(\R^d)} \Big\}
\end{eqnarray*}
with metric $d(u,w)=\sup_{0\leq t\leq T^*}\|u-w\|_{L^2}$, where $T^*$ is to be determined later.

Now we define a map $\mathcal T: \mathcal X \rightarrow \mathcal X$ as follows: For all $w\in \mathcal X$, let $u$ be the unique solution of the following Cauchy problem
\begin{eqnarray}\label{Line}
\partial_t u-a\Delta u +\nabla\cdot (u_{+,M} \nabla V_\varepsilon* w)=0, && x\in\R^d ,\\
u(x,0)=u_0(x),&&\nonumber
\end{eqnarray}
where $u_{+,M}=\min\{u_+,M\}$, $u_+=\max\{0,u\}$.
It is easy to see that the solution of \eqref{Line} has the property of conservation of mass, i.e.
\begin{eqnarray*}
\dint_{\R^d} u(x,t)dx =\dint_{\R^d} u_0(x)dx.
\end{eqnarray*}

Let $\alpha$ be an arbitrary multi-index with $|\alpha|\leq s$. Applying the operator $D^\alpha$ on both sides of equation \eqref{Line}, multiplying by $D^\alpha u$ and integrating on $\R^d$ leads to
\begin{eqnarray*}
&&\dfrac{1}{2}\dfrac{d}{dt}\dint_{\R^d}|D^\alpha u|^2dx +a\dint_{\R^d}|\nabla D^\alpha u|^2 dx\\[1ex]
&=& \dint_{\R^d} D^\alpha (u_{+,M}\nabla (V_\varepsilon*w))\cdot \nabla D^\alpha u dx\\[1ex]
&\leq & \|D^\alpha (u_{+,M}\nabla (V_\varepsilon*w))\|_{L^2}\|\nabla D^\alpha u\|_{L^2}\\[1ex]
&\leq& C\big(\|u_{+,M}\|_{L^\infty}\|D^s\nabla(V_\varepsilon *w)\|_{L^2}+\|D^su_{+,M}\|_{L^2}\|\nabla(V_\varepsilon *w)\|_{L^\infty}\big) \|\nabla D^\alpha u\|_{L^2}\\[1ex]
&\leq& \frac{a}{2}\|\nabla D^\alpha u\|_{L^2}^2 + C\big(\|u\|_{L^\infty}^2\|\nabla V_\varepsilon *D^s w\|_{L^2}^2+ \|D^su\|_{L^2}^2\|\nabla V_\varepsilon * w\|_{L^\infty}^2\big).
\end{eqnarray*}
where inequality \eqref{ineq1} is applied. By Hardy-Littlewood-Sobolev inequality and the Sobolev embedding $H^s\hookrightarrow L^\infty$, it follows
\begin{eqnarray*}
&&\dfrac{d}{dt}\dint_{\R^d}|D^\alpha u|^2dx +a\dint_{\R^d}|\nabla D^\alpha u|^2 dx\\[1ex]
&\leq &C(\|D^su\|^2_{L^2}\|\nabla V_\varepsilon\|_{L^1}^2 \|D^s w\|_{L^2}^2+ \|D^su\|_{L^2}^2\|\nabla V_\varepsilon\|_{L^1}^2 \|w\|_{L^\infty}^2)\\[1ex]
&\leq & C(\varepsilon) \|w\|_{L^\infty(0,T^*;H^s)}^2 \|D^s u\|_{L^2}^2.
\end{eqnarray*}
Taking the summation of all multi-index $|\alpha|\leq s$ on both sides, we get
\begin{eqnarray*}
&&\dfrac{d}{dt}\|u\|_{H^s}^2 +a\|u\|_{H^{s+1}}^2\leq C(\varepsilon, \tilde M) \|u\|_{H^s}^2.
\end{eqnarray*}
By Gronwall's inequality, we have
\begin{eqnarray*}
\sup_{0\leq t\leq T^*}\|u(\cdot,t)\|_{H^s}^2+a\|u\|_{L^2(0,T^*;H^{s+1})}^2\leq \|u_0\|_{H^s}^2 T^*e^{C(\varepsilon, \tilde M)  T^*}\leq \tilde M,
\end{eqnarray*}
where $T^*\leq T_1$ has been taken so small that $T_1e^{C(\varepsilon, \tilde M)  T_1}\leq 2$.

Next, we have to check that $0\leq u\leq M$. To do so, we use $u_-=-\min\{0,u\}$ as a test function, i.e.
\begin{eqnarray*}
\frac{1}{2}\dfrac{d}{dt}\dint_{\R^d} |u_-|^2dx +a\dint_{\R^d}|\nabla u_-|^2dx=0.
\end{eqnarray*}
Due to the non-negativity of the initial data $u_0$, we get $u(t,\cdot)\geq 0$ almost everywhere in $\R^d$. Similarly, using $u_M=(u-M)_+$ as a test function, we get
\begin{eqnarray*}
\frac{1}{2}\dfrac{d}{dt}\dint_{\R^d} |u_M|^2dx +a\dint_{\R^d}|\nabla u_M|^2dx=0.
\end{eqnarray*}
Due to the fact that $u_0\leq M$, we can conclude $u_M=0$ which means equivalently $u(t,\cdot)\leq M$ almost everywhere in $\R^d$.
In this way, we have built a map $\mathcal T$ from $\mathcal X$ to $\mathcal X$.

Now, we show that the map $\mathcal T$ is a contraction for a short time $T^*$ that depends on $\varepsilon$, $M$ and $\tilde M$.
Let $u_1=\mathcal T w_1$ and $u_2=\mathcal T w_2$, then we take the difference of the two equations, use $u_1-u_2$ as a test function and integrate on $\R^d$:
\begin{eqnarray*}
&&\frac{1}{2}\dfrac{d}{dt}\dint_{\R^d} |u_1-u_2|^2dx +a\dint_{\R^d}|\nabla (u_1-u_2)|^2dx\\
&=&\dint_{\R^d} (u_1-u_2)\nabla V_\varepsilon *w_1 \cdot\nabla (u_1-u_2)dx+\dint_{\R^d} u_2\nabla V_\varepsilon *(w_1-w_2) \cdot\nabla (u_1-u_2)dx\\
&\leq & \frac{a}{2}\dint_{\R^d}|\nabla (u_1-u_2)|^2dx+ C\|\nabla V_\varepsilon * w_1\|_{L^\infty}^2 \dint_{\R^d} |u_1-u_2|^2dx \\
&&\hspace{5cm}+C \|u_2\|_{L^\infty}^2 \dint_{\R^d} |\nabla V_\varepsilon * (w_1-w_2)|^2 dx.
\end{eqnarray*}
The Hardy-Littlewood-Sobolev inequality implies
\begin{eqnarray*}
&&\dfrac{d}{dt}\dint_{\R^d} |u_1-u_2|^2dx +a\dint_{\R^d}|\nabla (u_1-u_2)|^2dx\\
&\leq & C\|\nabla V_\varepsilon\|_{L^1}^2 \|w_1\|_{L^\infty}^2 \dint_{\R^d} |u_1-u_2|^2dx +C\|u_2\|_{L^\infty}^2 \|\nabla V_\varepsilon\|_{L^1}^2 \dint_{\R^d} |w_1-w_2|^2dx.
\end{eqnarray*}
Notice that $\|u_1(\cdot,0)-u_2(\cdot,0)\|_{L^2}=0$ and Gronwall's inequality leads to
\begin{eqnarray*}
\sup_{0\leq t\leq T^*} \|u_1(\cdot,t)-u_2(\cdot,t)\|_{L^2} \leq T^* e^{C(\varepsilon,M)T^*} \sup_{0\leq t\leq T^*} \|w_1-w_2\|_{L^2},
\end{eqnarray*}
which means for $T^*\leq T_1$ and $T^* e^{C(\varepsilon,M)T^*}\leq \frac12$, the map $\mathcal T$ is a contraction.

By Banach fixed-point theorem, the map $\mathcal T$ has a unique fixed-point in time interval $0\leq t\leq T^*(\varepsilon, M, \tilde M)$. Let $u^\varepsilon$ be the unique solution of
\begin{eqnarray*}
\partial_t u^\varepsilon-a\Delta u^\varepsilon +\nabla\cdot ((u^\varepsilon)_{+,M} \nabla V_\varepsilon* u^\varepsilon)=0, && x\in\R^d\, t\in (0,T^*),\\
u^\varepsilon(x,0)=u_0(x),&& 0\leq u_0(x)\leq M.
\end{eqnarray*}
Since we have shown that $0\leq u^\varepsilon\leq M$ in $ (0,T^*)\times\R^d$, we can replace $(u^\varepsilon)_{+,M}$ in the equation by $u^\varepsilon$ itself.

{\bf Step 2:} Global solution. According to the local existence result that we have obtained in step 1, there exists a maximum existence time $\hat T$ such that
\begin{eqnarray}\label{blowup}
\lim_{t\rightarrow\hat T}\|u(\cdot,t)\|_{H^s}=\infty.
\end{eqnarray}
With the help of $0\leq u\leq M$, we can show that the above blow-up criteria will not happen in finite time, which means that $\hat T=\infty.$ 
This can be again proved by using energy estimates for all $t<\hat T$ and any multi-index $\alpha$ with $|\alpha|\leq s$. Applying the operator $D^\alpha$ on both sides of equation \eqref{probinter}, multiplying by $D^\alpha u$ and integrating on $\R^d$ leads to
\begin{eqnarray*}
&&\dfrac{1}{2}\dfrac{d}{dt}\dint_{\R^d}|D^\alpha u|^2dx +a\dint_{\R^d}|\nabla D^\alpha u|^2 dx\\[1ex]
&=& \dint_{\R^d} D^\alpha (u_{+,M}\nabla (V_\varepsilon*u))\cdot \nabla D^\alpha u dx\\[1ex]
&\leq & \|D^\alpha (u\nabla (V_\varepsilon*u))\|_{L^2}\|\nabla D^\alpha u\|_{L^2}\\[1ex]
&\leq& C\big(\|u\|_{L^\infty}\|D^s\nabla(V_\varepsilon *u)\|_{L^2}+\|D^su\|_{L^2}\|\nabla(V_\varepsilon *u)\|_{L^\infty}\big) \|\nabla D^\alpha u\|_{L^2}\\[1ex]
&\leq& \frac{a}{2}\|\nabla D^\alpha u\|_{L^2}^2 + C\big(M^2\|\nabla V_\varepsilon *D^s u\|_{L^2}^2+ \|D^su\|_{L^2}^2\|\nabla V_\varepsilon * u\|_{L^\infty}^2\big)\\[1ex]
&\leq &\frac{a}{2}\|\nabla D^\alpha u\|_{L^2}^2 + 2M^2\|\nabla V_\varepsilon\|_{L^1}^2 \|D^s u\|_{L^2}^2.
\end{eqnarray*}
Taking the summation of all multi-index $|\alpha|\leq s$ on both sides, we get
\begin{eqnarray*}
&&\dfrac{d}{dt}\|u\|_{H^s}^2 +a\|u\|_{H^{s+1}}^2\leq C(\varepsilon, M) \|u\|_{H^s}^2.
\end{eqnarray*}
Gronwall's inequality leads to
\begin{eqnarray*}
\sup_{0\leq t<\hat T}\|u(\cdot,t)\|_{H^s}^2+a\|u\|_{L^2(0,\hat T;H^{s+1})}^2\leq \|u_0\|_{H^s}^2 e^{C(\varepsilon, M) \hat T}.
\end{eqnarray*}
If $\hat T$ is finite, the above estimate contradicts the blow-up criteria \eqref{blowup}. Therefore, the solution exists globally for any fixed $\varepsilon$.

{\bf Step 3:} Uniform in $\varepsilon$ estimates.
Let $u^\varepsilon\in \mathcal X$ be the solution of the following problem
\begin{eqnarray}\label{probepsilon}
\partial_t u^\varepsilon-a\Delta u^\varepsilon +\nabla\cdot (u^\varepsilon V_\varepsilon* \nabla u^\varepsilon)=0, && x\in\R^d , \\
u^\varepsilon(x,0)=u_0(x).&&\nonumber
\end{eqnarray}
The conservation of mass is satisfied, i.e.
\begin{eqnarray}\label{mass}
\dint_{\R^d} u^\varepsilon(x,t)dx =\dint_{\R^d} u_0(x)dx.
\end{eqnarray}

Multiplying the equation by $u^\varepsilon$, integrating on $\R^d$ and applying Hardy-Littlewood-Sobolev inequality yields
\begin{eqnarray*}
&&\dfrac{1}{2}\dfrac{d}{dt}\dint_{\R^d} |u^\varepsilon|^2 dx +a\dint_{\R^d}|\nabla u^\varepsilon|^2 dx =\dint_{\R^d} u^\varepsilon V_\varepsilon*\nabla u^\varepsilon \cdot\nabla u^\varepsilon dx\\
&\leq &(\frac{a}{2b}-\eta)\dint\dint_{\R^d\times\R^d}\Big|V_\varepsilon(x-y)\nabla u^\varepsilon(x,t) \nabla u^\varepsilon(y,t)\Big| dxdy\\
&\leq & (a-2b\eta)\dint_{\R^d}|\nabla u^\varepsilon|^2 dx.
\end{eqnarray*}
Therefore, we get that
$$
\sup_{ t\geq 0}\|u^\varepsilon(\cdot,t)\|_{L^2}^2+4b \eta \|\nabla u^\varepsilon\|_{L^2(0,\infty;L^2(\R^d))}^2\leq \|u_0\|_{L^2}^2.
$$ \end{proof}

The next theorem states a uniform estimate for the solution to the intermediate problem. 
\begin{thm}\label{estVsmallID}
Let $u^\varepsilon$ be the solution of \eqref{probinter}, then there exists a constant $K$ (depending on $s$ and $d$) such that for $\|u_0\|_{H^s}< \frac{a}{bK}$, the following uniform estimate in $\varepsilon$ holds
\begin{eqnarray}
 \sup_{t\geq 0}\|u^\varepsilon\|_{H^s(\R^d)} +\|Du^\varepsilon\|_{L^2(0,\infty;H^{s}(\R^d))}\leq C,
\end{eqnarray}
where $C$ is a constant independent of $\varepsilon$.
\end{thm}

\begin{proof}
Applying the differential operator $D^\alpha$ for the multi-index $|\alpha|\leq s$, multiplying by $D^\alpha u^\varepsilon$ and integrating over $\R^d$ leads to
\begin{eqnarray*}
&&\dfrac{1}{2}\dfrac{d}{dt}\dint_{\R^d} |D^\alpha u^\varepsilon|^2 dx +a\dint_{\R^d}|\nabla D^\alpha u^\varepsilon|^2 dx \\
&=& \dint_{\R^d} u^\varepsilon V_\varepsilon*\nabla D^\alpha u^\varepsilon \cdot\nabla D^\alpha u^\varepsilon dx +\dint_{\R^d} [D^\alpha, u^\varepsilon] \nabla (V_\varepsilon* u^\varepsilon)\nabla D^\alpha u^\varepsilon\\
&\leq & \|u^\varepsilon\|_{L^\infty}\dint\dint_{\R^d\times\R^d}\Big|V_\varepsilon(x-y)\nabla D^\alpha u^\varepsilon(x,t) \nabla D^\alpha u^\varepsilon(y,t)\Big| dxdy\\
&& +\big\|[D^\alpha, u^\varepsilon] \nabla (V_\varepsilon* u^\varepsilon)\big\|_{L^2} \big\|\nabla D^\alpha u^\varepsilon\big\|_{L^2}.
\end{eqnarray*}
The Hardy-Littlewood-Sobolev inequality and \eqref{ineq2} for the commutator $[D^\alpha, u^\varepsilon]$, we get
\begin{eqnarray*}
&&\dfrac{1}{2}\dfrac{d}{dt}\dint_{\R^d} |D^\alpha u^\varepsilon|^2 dx +a\dint_{\R^d}|\nabla D^\alpha u^\varepsilon|^2 dx \\
&\leq &\|u^\varepsilon\|_{L^\infty} \|V_\varepsilon\|_{L^1} \dint_{\R^d}|\nabla D^\alpha u^\varepsilon|^2 dx \\
&& + C\big(\|Du^\varepsilon\|_{L^\infty} \|D^{s-1}\nabla (V_\varepsilon * u^\varepsilon)\|_{L^2} +\|\nabla(V_\varepsilon*u^\varepsilon)\|_{L^\infty}\|D^s u^\varepsilon\|_{L^2}\big) \|\nabla D^\alpha u^\varepsilon\|_{L^2}\\[1ex]
&\leq & \|u^\varepsilon\|_{L^\infty} \|V_\varepsilon\|_{L^1}\|\nabla D^\alpha u^\varepsilon\|_{L^2}^2  + C \|V_\varepsilon\|_{L^1} \|Du^\varepsilon\|_{L^\infty}\|D^s u^\varepsilon\|_{L^2}   \|\nabla D^\alpha u^\varepsilon\|_{L^2},
\end{eqnarray*}
where the constant $C$ depends on $d$ and $s$.
The Gagliardo-Nirenberg-Sobolev inequality gives the following two estimates
\begin{eqnarray*}
\|D^s u\|_{L^2}\leq K\|u\|_{L^\infty}^\frac{2}{2s+2-d} \|D^{s+1}u\|_{L^2}^\frac{2s-d}{2s+2-d},\\[1ex]
\|D u\|_{L^\infty}\leq K\|u\|_{L^\infty}^\frac{2s-d}{2s+2-d} \|D^{s+1}u\|_{L^2}^\frac{2}{2s+2-d},
\end{eqnarray*}
where $K$ depends on $d$ and $s$.
Hence, we have the following estimate
\begin{eqnarray*}
&&\dfrac{d}{dt}\dint_{\R^d} |D^\alpha u^\varepsilon|^2 dx +2a\dint_{\R^d}|\nabla D^\alpha u^\varepsilon|^2 dx \leq  K \|V_\varepsilon\|_{L^1} \| u^\varepsilon\|_{L^\infty}\| D^{s+1} u^\varepsilon\|_{L^2}^2.
\end{eqnarray*}
After summing up the multi-index $|\alpha|\leq s$ and the use of the Sobolev embedding theorem $H^s\hookrightarrow L^\infty$, we end up with
\begin{eqnarray*}
\dfrac{d}{dt}\|u^\varepsilon\|_{H^s}^2 +2a\|\nabla u^\varepsilon\|_{H^s}^2 &\leq & K \|V_\varepsilon\|_{L^1} \|u^\varepsilon\|_{L^\infty}\|D^{s+1} u^\varepsilon\|_{L^2}^{2}\\
&\leq & K \|V_\varepsilon\|_{L^1} \|u^\varepsilon\|_{H^s}\|D^{s+1} u^\varepsilon\|_{L^2}^{2},
\end{eqnarray*}
where $\|V\|_{L^1}=2b$ and $K$ is a constant that only depends on $d$ and $s$. As a consequence, for initial data
$\|u_0\|_{H^s}< \frac{a}{bK}$, we obtain the global uniform estimate in $\varepsilon$, cf. \eqref{estVsmallID}.
\end{proof}

In the next subsection, we discuss the global solvability of the limiting problem to \eqref{probinter} for $\varepsilon\rightarrow 0$.

\subsection{Solvability of the limiting problem}
The limiting problem we are interested in is the following macroscopic diffusion-aggregation equation 
\begin{eqnarray}\label{eqn}
\partial_t u-\nabla(a\nabla u-2bu\nabla u)=0,&& x\in\R^d , \\
u(x,0)=u_0(x).\nonumber
\end{eqnarray}
Similar to our previous investigations, we study the existence and uniqueness of solutions to this equation. 
\begin{thm}\label{existWeak}
For any initial data $u_0\in L^1(\R^d)\cap L^\infty(\R^d)$ and $\|u_0\|_{L^\infty}<\frac{a}{2b}$, the Cauchy problem \eqref{eqn} has a non-negative weak solution in $L^\infty(0,\infty;L^1(\R^d)\cap L^\infty(\R^d))\cap L^2(0,\infty; H^1(\R^d))$ and
\begin{eqnarray}
\|u(\cdot,t)\|_{L^1(\R^d)}=\|u_0\|_{L^1(\R^d)}, \quad
\|u(\cdot,t)\|_{L^\infty(\R^d)}<\frac{a}{2b}, \label{eqnest1}\\ [1ex]
\sup_{t\geq 0}\|u(\cdot, t)\|_{L^2(\R^d)}+\|u\|_{L^2(0,\infty; H^1(\R^d))}\leq C(\|u_0\|_{L^2(\R^d)}).\label{eqnest2}
\end{eqnarray}

Furthermore, if $u_0\in H^s(\R^d)$ for $s>\frac{d}{2}+1$ and $\|u_0\|_{H^s}<\frac{a}{bK}$ (from theorem \ref{estVsmallID}), then for any given $T$, the solution $u\in L^\infty(0,T;L^1(\R^d)\cap H^s(\R^d))\cap L^2(0,T; H^{s+1}(\R^d))$ is unique and satisfies
\begin{eqnarray}
\sup_{0\leq t\leq T}\|u(\cdot, t)\|_{H^s(\R^d)} +\|u\|_{L^2(0,T; H^{s+1}(\R^d))}\leq C(\|u_0\|_{H^s(\R^d)}).\label{eqnest3}
\end{eqnarray}
\end{thm}

\begin{proof}
For any fixed time interval $[0,T]$, we know from Theorem \ref{thmexisV} that there exists a subsequence of $u^\varepsilon$ (without relabeling) such that
$$
u^\varepsilon\rightharpoonup u,\quad \mbox{ weakly in } L^2(0,T;H^1(\R^d)).
$$
Furthermore, due to the fact that $V_\varepsilon\rightarrow 2b\delta$ in the sense of distribution, we have that for any $\psi\in L^2(0,T;L^2(\R^d))$
\begin{eqnarray*}
&&\Big|\dint^T_0dt\dint_{\R^d}\psi(x,t) \dint_{\R^d}V_\varepsilon(x-y)\nabla u^\varepsilon(y,t)dxdy -2b\dint^T_0dt\dint_{\R^d}\psi(x,t) \nabla u(x,t)dx\Big|\\
\leq && \Big|\dint^T_0dt\dint_{\R^d} \Big(\dint_{\R^d}\psi(x,t)V_\varepsilon(x-y)dx-2b\psi(y,t)\Big) \nabla u^\varepsilon(y,t)dy\Big|\\
&&+2b\Big|\dint^T_0dt\dint_{\R^d}\psi(x,t) \nabla u^\varepsilon(x,t)dx-\dint^T_0dt\dint_{\R^d}\psi(x,t) \nabla u(x,t)dx\Big|
\rightarrow 0 \quad \mbox{ as }\varepsilon\rightarrow 0.
\end{eqnarray*}
Therefore,
$$
V_\varepsilon*\nabla u^\varepsilon\rightharpoonup 2b\nabla u,\quad \mbox{ weakly in } L^2(0,T;L^2(\R^d)).
$$

From the uniform estimates in $\varepsilon$ (see Theorem \ref{thmexisV}), we can deduce the estimate for the aggregation term by using Hardy-Littlewood-Sobolev inequality, i.e.
\begin{eqnarray}\label{estagg}
\|u^\varepsilon V_\varepsilon *\nabla u^\varepsilon\|_{L^2(0,\infty;L^2(\R^d))}\leq \|u^\varepsilon\|_{L^\infty(0,\infty;L^\infty(\R^d))}\|V_\varepsilon\|_{L^1(\R^d)}\|\nabla u^\varepsilon\|_{L^2(0,\infty;L^2(\R^d))}\leq C,
\end{eqnarray}
from which we obtain the uniform estimate for the time derivative
$$
\|\partial_t u^\varepsilon\|_{L^2(0,T;H^{-1}(\R^d))}=\|\nabla\cdot(a\nabla u^\varepsilon-2b u^\varepsilon V_\varepsilon *\nabla u^\varepsilon)\|_{L^2(0,T;H^{-1}(\R^d))}\leq C.
$$

For a sequence of balls $B_{R_k}\in \R^d$ with radius $R_k\rightarrow \infty$ ($k\rightarrow \infty$), there exists a subsequence that strongly converges in $L^2(0,T;L^2(B_{R_k}))$
due to the compact embedding $H^1(\R^d)\hookrightarrow\hookrightarrow L^2(\R^d)$ and Aubin-Lions lemma (for example in \cite{ChenJungel,Simon87}). 
After a standard diagonal argument, we obtain a subsequence of $u^\varepsilon$ (again without relabeling) such that for any bounded ball $B_R\subset \R^d$
$$
u^\varepsilon\rightarrow u\quad \mbox{ strongly in } L^2(0,T;L^2(B_R)).
$$
For the aggregation term, we have that
$$
u^\varepsilon V_\varepsilon*\nabla u^\varepsilon \rightharpoonup 2bu\nabla u \quad \mbox{ weakly in } L^1(0,T;L^1(B_R)).
$$
Together with the estimate in \eqref{estagg}, we get that the above weak convergence is in $L^2(0,T;L^2(B_R))$.
Thus, for any test function $\varphi\in C^\infty_0(\R^d)$, $\eta\in C^\infty([0,T])$, we can take the limit in the following weak formulation of the intermediate problem
$$
\dint^T_0 \langle \partial_t u^\varepsilon,\varphi\rangle_{\langle H^1,H^{-1}\rangle } \eta(t)dt =\dint^T_0 \dint_{\R^d} (a\nabla u^\varepsilon-2b u^\varepsilon V_\varepsilon\nabla u^\varepsilon)\cdot\nabla\varphi  dx \eta(t) dt
$$
and obtain that $u$ is a weak solution to the limiting problem.

The estimates in \eqref{eqnest1}, \eqref{eqnest2} and \eqref{eqnest3} follow directly from the uniform estimates in Theorem \ref{thmexisV} and \ref{estVsmallID}.

In the last step, we prove the uniqueness of the solution. We assume that $u_1$ and $u_2$ are two solutions with the same initial data $u_0$. 
The difference $u_1-u_2$ is then used as a test function in the weak solution formulation
\begin{eqnarray*}
&&\frac{1}{2}\dfrac{d}{dt}\int_{\R^d} |u_1-u_2|^2 dx \\
&\leq & \int_{\R^d} -(a-2bu_1) |\nabla (u_1-u_2)|^2 dx +\dint_{\R^d} (u_1-u_2)\nabla u_2\cdot \nabla (u_1-u_2)\\
&\leq & -\frac{\eta}{2}\int_{\R^d} |\nabla (u_1-u_2)|^2dx +\| \nabla u_2\|_{L^\infty(0,T;H^s(\R^d))} \dint_{\R^d} |u_1-u_2|^2 dx,
\end{eqnarray*}
from where it follows that
$$
\dsup_{0\leq t\leq T} \|(u_1-u_2)(\cdot,t)\|_{L^2(\R^d)}\leq e^{CT} \|u_0-u_0\|_{L^2(\R^d)}=0.
$$ \end{proof}

We remark that in section \ref{sec:num_res}, we analyze the condition $\|u_0\|_{L^\infty}<\frac{a}{2b}$ from a numerical point of view.
That means, we study the expression $a = 2b||u_0||_{L^\infty}\eta$ for $\eta \geq 0$, where $\eta>1$ identifies the diffusion
and $\eta <1$ the aggregation regime. In particular, the case $\eta=1$ is carefully evaluated.

\subsection{Estimate for $u^\varepsilon-u$}
To finish our investigations on the solvability of the intermediate and macroscopic problem,
we give an estimate for the difference of their solutions. 

\begin{lem}\label{lemuuep}
Let $u$ and $u^\varepsilon$ be the solutions of \eqref{eqn} and \eqref{probinter} with the same initial data $u_0$ and uniform estimates in $L^\infty(0,T;H^s(\R^d))\cap L^2(0,T;H^{s+1}(\R^2))$.
Let $V\in C^2_0(\R^d)$ equipped with compact support $B_1$, then the following estimate holds
\begin{eqnarray*}
\|u^\varepsilon-u\|_{L^\infty(0,T;L^2(\R^d))} +\|\nabla(u^\varepsilon-u)\|_{L^(0,T;L^2(\R^d))}\leq C(T)\varepsilon.
\end{eqnarray*}
\end{lem}

\begin{proof}
Taking the difference of the two equations \eqref{eqn} and \eqref{probinter}, we obtain
\begin{eqnarray*}
&&\partial_t(u^\varepsilon-u) -\nabla \cdot \Big(a\nabla (u^\varepsilon-u) -(u^\varepsilon-u)V_\varepsilon*\nabla u^\varepsilon \\
&& \hspace{2cm}- u(V_\varepsilon *\nabla(u^\varepsilon-u)) -u(V_\varepsilon *\nabla u -2b\nabla u)\Big) =0.
\end{eqnarray*}
Multiplying by $u^\varepsilon-u$ and integrating on $\R^d$ leads to
\begin{eqnarray*}
&&\frac{1}{2}\dfrac{d}{dt}\dint_{\R^d} |u^\varepsilon-u|^2 dx + a\int_{\R^d} |\nabla (u^\varepsilon-u)|^2 dx -\int_{\R^d} (u^\varepsilon -u) V_\varepsilon *\nabla u^\varepsilon \cdot \nabla (u^\varepsilon-u) dx\\
&&\hspace{5mm} -\dint_{\R^d} u(V_\varepsilon*\nabla (u^\varepsilon-u))\cdot\nabla (u^\varepsilon-u)dx -\dint_{\R^d} u (V_\varepsilon *\nabla u-2b\nabla u)\cdot \nabla (u^\varepsilon-u) =0,
\end{eqnarray*}
from where we obtain
\begin{eqnarray*}
&&\frac{1}{2}\dfrac{d}{dt}\dint_{\R^d} |u^\varepsilon-u|^2 dx + a\int_{\R^d} |\nabla (u^\varepsilon-u)|^2 dx\\
&\leq & \|u^\varepsilon-u\|_{L^2(\R^d)} \|V_\varepsilon\|_{L^1(\R^d)} \|\nabla u^\varepsilon\|_{L^\infty(\R^d)} \|\nabla (u^\varepsilon-u)\|_{L^2(\R^d)} \\[1ex]
& & +\|u\|_{L^\infty(\R^d)} \|V_\varepsilon\|_{L^1(\R^d)} \|\nabla (u^\varepsilon-u)\|_{L^2(\R^d)} \|\nabla (u^\varepsilon-u)\|_{L^2(\R^d)} \\[1ex]
& & +\|u\|_{L^\infty(\R^d)} \|V_\varepsilon*\nabla u- 2b \nabla u\|_{L^2(\R^d)} \|\nabla (u^\varepsilon-u)\|_{L^2(\R^d)}.
\end{eqnarray*}
Due to the fact that $\forall g\in L^2(\R^d)$, it holds
\begin{eqnarray*}
&&\Big|\dint_{\R^d}\dint_{\R^d} V_\varepsilon(x-y) (\nabla u(y)-\nabla u(x))g(x) dydx\Big|\\
&\leq & \varepsilon \Big|\int^1_0 \dint_{\R^d}\dint_{\R^d} |V_\varepsilon (z)| \cdot |D^2u(rz+y)| g(y+z) dydz \Big|\\[1ex]
&\leq & \varepsilon \|D^2 u\|_{L^2(\R^d)} \|V_\varepsilon\|_{L^1(\R^d)} \|g\|_{L^2(\R^d)},
\end{eqnarray*}
which means
$$
\|V_\varepsilon*\nabla u- 2b \nabla u\|_{L^2(\R^d)}\leq 2b\varepsilon  \|D^2 u\|_{L^2(\R^d)}.
$$

Since $\|u\|_{L^\infty(\R^d)} \|V_\varepsilon\|_{L^1(\R^d)}\leq (\frac{a}{2b}-\eta)\cdot 2b=a-2b\eta$ and after using Young's inequality, we end up  with
\begin{eqnarray*}
&&\frac{1}{2}\dfrac{d}{dt}\dint_{\R^d} |u^\varepsilon-u|^2 dx + b\eta\int_{\R^d} |\nabla (u^\varepsilon-u)|^2 dx\\
&\leq & C \dint_{\R^d} |u^\varepsilon-u|^2dx + C\varepsilon^2 \|D^2u(\cdot,t)\|^2_{L^2(\R^d)}.
\end{eqnarray*}
Hence, the desired estimates are obtained by Gronwall's inequality together with taking the same initial data $u(x,0)=u^\varepsilon(x,0)=u_0(x)$.\end{proof}

\section{Rigorous derivation of the mean-field limit}\label{sec:RigorousMeanField}
In this section, we assume that the solutions for the intermediate problem \eqref{probinter} and the limiting problem \eqref{eqn} exist uniquely and satisfy the necessary {\it a priori} estimates that are needed in deriving the mean-field limit. 
Then, starting from the stochastic particle system \eqref{eq:micro}, we rigorously derive the
diffusion-aggregation equation \eqref{eqn} by exploiting the intermediate particle system with smooth interaction potential \eqref{probinter}.
The unique existence and the corresponding estimates can be obtained, for example, by Theorems \ref{thmexisV}, \ref{estVsmallID} and \ref{existWeak}.

\subsection{Stochastic particle systems}
In the following we use $(B^i(t))_{1\leq i\leq N}$ as a set of independent standard Brownian motions for each particle $i.$
The discrete particle model reads
\begin{eqnarray}\label{particle}
&&d X^i_{\varepsilon,N}(t)= \sqrt{2a}\,dB^i(t) +\frac{1}{N}\sum_{j\neq i} \nabla V_\varepsilon(|X^i_{\varepsilon,N}(t)-X^j_{\varepsilon,N}(t)|)dt,
\end{eqnarray}
where $V_\varepsilon(r)=\varepsilon^{-d}V(r/\varepsilon)$ and $\dint_{\R^d}V(x)dx=2b$.
The corresponding initial data is given by
\begin{eqnarray}\label{iniparticle}
X^i_{\varepsilon,N}(0)=\xi^i,\quad \mbox{ where $\xi^i$ are $N$ i.i.d random variables with pdf $u_0(x)$.}
\end{eqnarray}

Since for fixed $\varepsilon$, the gradient $\nabla V_\varepsilon$ is bounded Lipschitz continuous, 
we can use the following result for the unique solvability of initial value problems for stochastic particle systems:
\begin{lem} For any fixed $\varepsilon$, 
the problem \eqref{particle}-\eqref{iniparticle} has a unique global solution $X^i_{\varepsilon, N}(t)$.
\end{lem}

We note that the trajectory of the intermediate problem \eqref{probinter} is
\begin{eqnarray}\label{intparticle}
&&d \bar X^i_{\varepsilon}(t)= \sqrt{2a}\,dB^i(t) +\int_{\R^d}\nabla V_\varepsilon(|\bar X^i_{\varepsilon}(t)-y|) u^\varepsilon(y,t) dy dt,
\end{eqnarray}
where $u^\varepsilon(x,t)$ is the probability density function of random variables $\bar X^i_{\varepsilon}(t)$,
and 
the trajectory of the limiting problem \eqref{eqn} is
\begin{eqnarray}
\label{SDEs}
d\hat X^{i}(t) &=& \sqrt{2a}\,dB^i(t) -2b\nabla u(\hat X^i(t),t)dt.
\end{eqnarray}

In order to compare the three problems \eqref{particle}, \eqref{intparticle} and \eqref{SDEs}, we take the same initial data \eqref{iniparticle}
for $X^i_{\varepsilon,N}(0), \bar X^i_\varepsilon(0)$ and $\hat X^i(0),$ i.e.
$$
X^i_0=\xi_i \quad \mbox{ i.i.d. random variables with pdf } u_0.
$$

With the help of the unique solvability of the problems investigated in section \ref{sec:solva}, we also have the existence and uniqueness of the initial value problems of the intermediate and the limiting trajectory. Namely,

\begin{lem}
If \eqref{probinter} has a unique solution $u^\varepsilon$ with $\nabla u^\varepsilon\in L^\infty(0,+\infty;W^{1,\infty}(\R^d))$, then the initial value problem \eqref{intparticle},\eqref{iniparticle} has a unique global solution $(\bar X^i_\varepsilon(t),u^\varepsilon(x,t))$.
\end{lem}

\begin{proof} Let $v$ be the solution of \eqref{probepsilon} which satisfies the initial data $v(x,0)=u_0(x)$. By assumption, we have that $\nabla V_\varepsilon *v=V_\varepsilon*\nabla v$ is a bounded Lipschitz function. Therefore, the initial value problem
\begin{eqnarray*}
&& d\bar X_\varepsilon(t)=\sqrt{2a}\, dB(t) +(\nabla V_\varepsilon*v)(\bar X_\varepsilon(t))dt,\\[1ex]
&& \bar X(0)=\xi\quad \mbox{ given random variable with pdf $u_0(x)$}
\end{eqnarray*}
has a unique global solution $\bar X_\varepsilon(t)$. Let $u^\varepsilon$ be the probability density function. 
Then, we have from It$\hat o$'s formula
for any smooth test function $\varphi(x,t)$ that
\begin{eqnarray*}
\varphi(\bar X_\varepsilon(t),t)-\varphi(\xi,0) &=& \int^t_0 \Big[\partial_t\varphi(\bar X_\varepsilon(s),s) +(V_\varepsilon*\nabla v)(\bar X_\varepsilon(s),s)\cdot \nabla\varphi(\bar X_\varepsilon(s),s))\\
&& +a\Delta\varphi(\bar X_\varepsilon(s),s)\Big]ds +\sqrt{2a}\int^t_0\nabla\varphi(\bar X_\varepsilon(s),s)dB_s.
\end{eqnarray*}
By taking the expectation, we get
\begin{eqnarray*}
&&\int_{\R^d} u^\varepsilon(x,t)\varphi(x,t)dx -\int_{\R^d}u_0(x)\varphi(x,0) dx\\
&=&\int^t_0\int_{\R^d}u^\varepsilon(x,s) \Big(\partial_t\varphi(x,s)+\nabla V_\varepsilon*v(x,s)\cdot\nabla \varphi(x,s)+a\Delta\varphi(x,s)\Big)dxds
\end{eqnarray*}
which is exactly the weak formulation of \eqref{probepsilon} with $v=u^\varepsilon$. By the assumption that the solution to this problem exists uniquely, we obtain that the unique solution is $u^\varepsilon$, i.e. the probability density of $\bar X_\varepsilon$. In other words, the unique solution of 
\eqref{intparticle},\eqref{iniparticle} is given by $(\bar X_\varepsilon, u^\varepsilon)$.
\end{proof}

By the same method, it can be easily shown that the initial value problem of the limiting trajectory is also uniquely solvable.

\begin{lem}
If \eqref{eqn} has a unique solution $u$ with $\nabla u\in L^\infty(0,+\infty;W^{1,\infty}(\R^d))$, then the initial value problem \eqref{SDEs},\eqref{iniparticle} has a unique global solution $(\bar X^i(t),u(x,t))$.
\end{lem}

\subsection{Convergence estimate for $N\rightarrow\infty$}
As a next step, we follow the ideas in \cite{JM98,Oel90,Phi07} to show the convergence in the large particle case. 
With the help of the existence theory and the estimates derived in section \ref{sec:solva}, 
we detect that some of the error estimates are different from those in the porous medium context, cf. \cite{JM98,Oel90,Phi07}. 
Therefore, for completeness, we give details of the proof. 

Let $V\in C^2_0(\R^d)$ and without loss of generality, let the compact support of $V$ be the unit ball. Thus, we have ${\rm supp} V_\varepsilon=B_\varepsilon(0)$.
The first lemma determines an estimate for the difference of the particle system and the intermediate problem.

\begin{lem}\label{lemXNbar}
 For any fixed $0<\delta\ll 1$ and time $t>0$, let $\varepsilon$ such that ${\frac{1}{\varepsilon^{2d+4}}}\leq \delta\ln N$, then 
$$
\E\Big(\sup_{0\leq s\leq t}\sup_{i=1,\cdots,N} \big|X^i_{\varepsilon,N}(s)-\bar X^i_\varepsilon(s)\big|^2\Big)\leq \frac{C(t)}{N^{1-C(t)\delta}},
$$
where $C(t)$ is a constant only depending on $t$, $\|V''\|_{L^\infty}$ and $\|u^\varepsilon\|_{L^\infty(0,\infty;H^s(\R^d))}$.
\end{lem}

\begin{proof} The fact $\|V_\varepsilon\|_{W^{2,\infty}}\leq \frac{1}{\varepsilon^{d+2}}\|V''\|_{\infty}$ is used within the proof several times.
Let
$$
S(t)=\sup_{i=1,\cdots,N} \big|X^i_{\varepsilon,N}(s)-\bar X^i_\varepsilon(s)\big|^2.
$$
By taking the difference of the two problems \eqref{particle} and \eqref{intparticle}, we obtain
\begin{eqnarray*}
&&\sup_{i=1,\cdots,N} \big|X^i_{\varepsilon,N}(t)-\bar X^i_\varepsilon(t)\big|^2\\
&\leq &\int^t_0\dfrac{t}{N^2}\sup_{i=1,\cdots,N} \Big|\sum^N_{l=1}\Big(\nabla V_\varepsilon(X^i_{\varepsilon,N}(s)-X^l_{\varepsilon,N}(s))-\nabla V_\varepsilon* u^\varepsilon(\bar X^i_\varepsilon(s),s)\Big)\Big|^2ds.
\end{eqnarray*}
Applying the expectation leads to 
\begin{eqnarray*}
\E(S(t))&\leq &\int^t_0\dfrac{t}{N^2}\E\Big(\sup_{i=1,\cdots,N} \Big|\sum^N_{l=1}\Big(\nabla V_\varepsilon(X^i_{\varepsilon,N}(s)-X^l_{\varepsilon,N}(s))-\nabla V_\varepsilon* u^\varepsilon(\bar X^i_\varepsilon(s),s)\Big)\Big|^2\Big)ds\\
&\leq &\dfrac{t}{N^2}\int^t_0 \Big\{\E\Big(\sup_{i=1,\cdots,N} \Big|\sum^N_{l=1}\Big(\nabla V_\varepsilon(X^i_{\varepsilon,N}(s)-X^l_{\varepsilon,N}(s))-\nabla V_\varepsilon(X^i_{\varepsilon,N}(s)-\bar X^l_{\varepsilon}(s))\Big)\Big|^2\Big)\\
&& \quad +\E\Big(\sup_{i=1,\cdots,N} \Big|\sum^N_{l=1}\Big(\nabla V_\varepsilon(X^i_{\varepsilon,N}(s)-\bar X^l_{\varepsilon}(s))-\nabla V_\varepsilon(\bar X^i_{\varepsilon}(s)-\bar X^l_{\varepsilon}(s))\Big)\Big|^2\Big)\\
&&\quad +\E\Big(\sup_{i=1,\cdots,N} \Big|\sum^N_{l=1}\Big(\nabla V_\varepsilon(\bar X^i_{\varepsilon}(s)-\bar X^l_\varepsilon(s))
-\nabla V_\varepsilon* u^\varepsilon(\bar X^i_\varepsilon(s),s)\Big)\Big|^2\Big)\Big\}ds\\
&=& I_1+I_2+I_3.
\end{eqnarray*}
Now, we derive the estimates for $I_1$, $I_2$ and $I_3$ separately.
\begin{eqnarray*}
|I_1|&\leq & \frac{t}{N^2}\int^t_0 \dfrac{\|V''\|_\infty^2}{\varepsilon^{2d+4}} \E \Big(\big(\sum^N_{l=1}\big|X^l_{\varepsilon,N}(s)-\bar X^l_\varepsilon(s)\big|\big)^2\Big)ds\\
&\leq & \frac{t\|V''\|_{\infty}^2}{\varepsilon^{2d+4}}\int^t_0\E \Big(\sup_{l=1,\cdots,N}\big|X^l_{\varepsilon,N}(s)-\bar X^l_\varepsilon(s)\big|^2\Big)ds\\
&\leq & \frac{t\|V''\|_{\infty}^2}{\varepsilon^{2d+4}}\int^t_0\E (S(s))ds.
\end{eqnarray*}
The second term can be handled similarly,
\begin{eqnarray*}
|I_2|&\leq & \frac{t}{N^2}\int^t_0 \dfrac{\|V''\|_\infty^2}{\varepsilon^{2d+4}} \E \Big(\sup_{i=1,\cdots,N}\big(N\big|X^i_{\varepsilon,N}(s)-\bar X^i_\varepsilon(s)\big|\big)^2\Big)ds\\
&\leq & \frac{t\|V''\|_{\infty}^2}{\varepsilon^{2d+4}}\int^t_0\E \Big(\sup_{i=1,\cdots,N}\big|X^i_{\varepsilon,N}(s)-\bar X^i_\varepsilon(s)\big|^2\Big)ds\\
&\leq & \frac{t\|V''\|_{\infty}^2}{\varepsilon^{2d+4}}\int^t_0\E (S(s))ds.
\end{eqnarray*}
The third term is estimated as follows
\begin{eqnarray*}
|I_3|&\leq & \frac{t}{N^2}\int^t_0 \E \Big[\sup_{i=1,\cdots,N}\sum^N_{l=1}\big(\nabla V_\varepsilon(\bar X^i_{\varepsilon}(s)-\bar X^l_\varepsilon(s))-\nabla V_\varepsilon*u^\varepsilon(\bar X^i_\varepsilon(s),s))\\
&&\hspace{2cm} \sum^N_{m=1}\big(\nabla V_\varepsilon(\bar X^i_{\varepsilon}(s)-\bar X^m_\varepsilon(s))-\nabla V_\varepsilon*u^\varepsilon(\bar X^i_\varepsilon(s),s))\Big]ds\\
&= &\frac{t}{N^2} \sum^N_{l=1}\sum^N_{m=1}\int^t_0\E \Big[\sup_{i=1,\cdots,N}\big(\nabla V_\varepsilon(\bar X^i_{\varepsilon}(s)-\bar X^l_\varepsilon(s))-\nabla V_\varepsilon*u^\varepsilon(\bar X^i_\varepsilon(s),s))\\
&&\hspace{2cm} \big(\nabla V_\varepsilon(\bar X^i_{\varepsilon}(s)-\bar X^m_\varepsilon(s))-\nabla V_\varepsilon*u^\varepsilon(\bar X^i_\varepsilon(s),s))\Big]ds,
\end{eqnarray*}
where for $l\neq m$ the expectation is zero. Hence,
$$
|I_3|\leq \frac{t}{N^2}\sum^N_{l=1}\int^t_0 \E\Big[\sup_{i=1,\cdots,N}\Big(\nabla V_\varepsilon(\bar X^i_{\varepsilon}(s)-\bar X^l_\varepsilon(s))-\nabla V_\varepsilon*u^\varepsilon(\bar X^i_\varepsilon(s),s)\Big)^2\Big]ds\leq \dfrac{Ct^2}{N},
$$
while exploiting the fact that $\|\nabla V_\varepsilon * u^\varepsilon\|_{L^\infty}\leq \| V_\varepsilon\|_{L^1}\|\nabla u^\varepsilon\|_{L^\infty} \leq C\|u^\varepsilon\|_{H^s}\leq C$.

Summarizing, we end up with
$$
\E(S(t))\leq \dfrac{2||V^{\prime \prime}||_\infty^2 t}{\varepsilon^{2d+4}}\int^t_0 \E(S(s))ds +\dfrac{Ct^2}{N},
$$
from which we obtain
\begin{align}
\E(S(t))\leq t\frac{C}{N}\left(\frac{\sqrt{\pi}}{2}\frac{\varepsilon^{d+2}}{||V^{\prime \prime}||_\infty}e^{t^2 ||V^{\prime \prime}||_\infty^2 \varepsilon^{-2d-4}}+t\right). \label{eq:Estimate}
\end{align}
Now, for any fixed $0<\delta\ll 1$, we can choose $\varepsilon$ so small that $e^{\frac{1}{\varepsilon^{2d+4}}}\leq N^\delta$. 
By taking the supremum in time on both sides, we have the conclusion.
\end{proof}
Note that the estimate \eqref{eq:Estimate} plays an important role for the numerical simulation of the
stochastic particle system in section \ref{sec:num_res} to determine a valid number of particles.

The next lemma intends to give an estimate on the difference of the intermediate and limiting problem.
\begin{lem} \label{lemXbarhat}Let $s>\frac{d}{2}+2$ and any fixed time $t>0$, then
$$
\E \big(\sup_{0\leq s\leq t}\big|\bar X_{\varepsilon}(s)-\hat  X(s)\big|\big)\leq C(t)\varepsilon.
$$
where $C(t)$ is a constant only depending on $t$, $\|V''\|_{L^\infty}$ and $\|u^\varepsilon\|_{L^\infty(0,\infty;H^s(\R^d))}$,
\end{lem}

\begin{proof}
Taking the difference between the intermediate \eqref{intparticle} and the limiting problem \eqref{SDEs} and considering
$$
\J(t)=\big|\bar X_{\varepsilon}(t)-\hat X(t)\big|,
$$
allows for the following representation:
\begin{eqnarray*}
\E(\J(t))&\leq & \dint^t_0\E\Big(\Big|2b\nabla u (\hat X(s),s) -\nabla V_\varepsilon * u^\varepsilon(\bar X_\varepsilon(s),s)\Big|\Big) ds\\
&\leq & \dint^t_0\E\Big(\Big|2b\nabla u (\hat X(s),s) -V_\varepsilon * \nabla u(\hat X(s),s)\Big|\Big) ds\\
&& +\dint^t_0\E\Big(\Big|V_\varepsilon * \nabla u(\hat X(s),s) -V_\varepsilon * \nabla u^\varepsilon(\hat X(s),s)\Big|\Big) ds\\
& & +\dint^t_0\E\Big(\Big|V_\varepsilon * \nabla u^\varepsilon(\hat X(s),s) - V_\varepsilon * \nabla u^\varepsilon(\bar X_\varepsilon(s),s)\Big|\Big) ds\\
&=& J_1+J_2+J_3.
\end{eqnarray*}
The estimate for $J_1$ is
\begin{eqnarray*}
J_1 &=&\dint^t_0\dint_{\R^d}\Big|\dint_{\R^d}V_\varepsilon(x-y)(\nabla u(x,s)-\nabla u(y,s)) dy \Big| u(x,s)dx ds \\
&=&\dint^t_0\dint_{\R^d}\Big|\int^1_0dr\dint_{\R^d}V_\varepsilon(x-y)\partial_r\nabla u(rx+(1-r)y,s)dy dr\Big| u(x,s)dx ds \\
&\leq & \varepsilon \dint^t_0\dint^1_0\dint_{\R^d}\dint_{\R^d}|V_\varepsilon(x-y)|\cdot|D^2 u(rx+(1-r)y,s) | u(x,s)dy dx drds \\
&= &\varepsilon \dint^t_0\dint^1_0\dint_{\R^d}\dint_{\R^d}|V_\varepsilon(z)|\cdot|D^2 u(rz+y,s) | \cdot |u(y+z,s)|dy dz drds \\
&\leq & 2b\varepsilon \dint^t_0 \|D^2u(\cdot,s)\|_{L^2(\R^d)}^2 \|u(\cdot,s)\|_{L^2(\R^d)}^2 ds\\[1ex]
&\leq & 2b\varepsilon \|u\|_{L^\infty(0,t;L^2(\R^d))}^2 \|D^2u\|_{L^2(0,t;L^2(\R^d))}^2.
\end{eqnarray*}
The expression $J_2$ can be estimated with the help of Lemma \ref{lemuuep}:
\begin{eqnarray*}
J_2 &=&\dint^t_0\dint_{\R^d}\Big|\dint_{\R^d}V_\varepsilon(x-y)(\nabla u(y,s)-\nabla u^\varepsilon(y,s)) dy \Big| u(x,s)dx ds \\
&\leq &\|V_\varepsilon\|_{L^1(\R^d)}\dint^t_0 \|(\nabla u-\nabla u^\varepsilon)(\cdot,s)\|_{L^2(\R^d)} \|u(\cdot,s)\|_{L^2(\R^d)}ds
\\
&\leq & 2b\|\nabla u-\nabla u^\varepsilon\|_{L^2(0,t;L^2(\R^d))} \|u\|_{L^2(0,t;L^2(\R^d))}\\[1ex]
&\leq & C\varepsilon.
\end{eqnarray*}
Finally, the estimate for $J_3$ is
\begin{eqnarray*}
J_3 &\leq & \|V_\varepsilon*D^2 u^\varepsilon\|_{L^\infty}\dint^t_0 \E (\J(s))ds \leq C \dint^t_0 \E (\J(s))ds.
\end{eqnarray*}
Then, by Gronwall's inequality, we get
$$
\E(\J_t)\leq C(t) \varepsilon.
$$
The conclusion is obtained by taking the supremum in time on both sides.
\end{proof}

Collecting the results from Lemma \ref{lemXNbar} and \ref{lemXbarhat} combined with the existence result in section \ref{sec:solva},
we are able to state the main theorem of this section on the mean-field limit.
\begin{thm} Assume $u_0\in L^1(\R^d)\cap H^s(\R^d)$ for $s>\frac{d}{2}+2$ and $\|u_0\|_{H^s}<\frac{a}{bK}$, then for ${\frac{1}{\varepsilon^{2d+4}}}\leq \delta\ln N$ it holds that
\begin{eqnarray*}
\E\Big(\sup_{0\leq s\leq t}\sup_{i=1,\cdots,N} \big|X^i_{\varepsilon,N}(s)-\hat X^i(s)\big|^2\Big)\leq C(t)\varepsilon^2 ,
\end{eqnarray*}
where $C(t)$ is a constant only depending on $t$, $\|V''\|_{L^\infty}$ and $\|u_0\|_{H^s(\R^d)}$,
\end{thm}

\begin{rem} 
Using the results on the convergence of trajectories, we can also get the corresponding propagation of chaos results (which means that the empirical measure $\frac{1}{N}\sum^N_{i=1}\delta_{X^i_{\varepsilon,N}(t)}$ converges weakly to the measure with probability density $u(x,t)$), see for example Theorem 3.1 in \cite{Phi07}.
\end{rem}

The next sections are devoted to the numerical investigations of the diffusion-aggregation problem \eqref{eqn}
and comparisons to the stochastic particle system.

\section{Numerical discretization schemes}
\label{sec:disc}
Starting from the stochastic particle system \eqref{particle}-\eqref{iniparticle}, 
we introduce a straightforward discretization and explain the numerical implementation.
We also develop a suitable, positive preserving discretization scheme for the diffusion-aggregation problem \eqref{eqn}.
Numerical results are then discussed in section \ref{sec:num_res}.

\subsection{Discretization of the stochastic particle system} 
\label{subsec:ParticleSystem}
To approximate the stochastic particle model \eqref{particle}-\eqref{iniparticle},
we use the Euler-Maruyama method, see for example \cite{KloedenPlaten1992}. Let $\{0=t_0<t_1<\cdots<t_S=T\}$ be a time discretization of $[0,T]$ 
for some $T\geq 0$. Furthermore, let $\Delta t_n = t_{n+1}-t_n$ be the corresponding step-sizes and $\Delta B^i_n = B^i(t_{n+1})-B^i(t_{n})$ the Brownian increments for $n= 0,\dots,S-1$. We denote by $Y^i_n$ the approximated solution of the system \eqref{particle}-\eqref{iniparticle}
at time $t_n$ satisfying
\begin{align}
Y^i_{n+1} &= Y^i_n + \Delta B^i_n \sqrt{2a}+\Delta t_n \frac{1}{N}\sum_{j\neq i} \nabla V_\varepsilon(|Y^i_n-Y^j_n|),\label{eq:Euler1}\\
Y^i_0 &= \xi_i,\label{eq:Euler2}
\end{align} 
for every $i=1,\dots,N$ and $n=0,\dots,S-1$. The sequence $(Y_n)_n$ of random variables is called Euler-Maruyama approximation for the initial value problem 
\eqref{particle}-\eqref{iniparticle}.

We aim to analyze the behavior of the particle system when the initial values $\xi_i$ are i.i.d.
and the latter distribution is given by a density which is a superposition of Barenblatt profiles.
We choose the following normalized Barenblatt profile as a basic component:
\begin{align}
\tilde{B}_{T,x_0}(x) := \frac{\sqrt{3}}{8}\left(T^{\frac{2}{3}}-\frac{(x-x_0)^2}{12}\right)_+ ,
\end{align}
which satisfies $||\tilde{B}_{T,x_0}||_{L^1(\R)} = 1$, $\operatorname{supp}(\tilde{B}_{T,x_0}) = [x_0-\sqrt{12T^{\frac{2}{3}}},x_0+\sqrt{12T^{\frac{2}{3}}}]$ and $||\tilde{B}_{T,x_0}||_\infty = \frac{\sqrt{3}}{8}T^{\frac{2}{3}}$.
We set \[u_0(x) = \sum_{l=1}^m \alpha_l \beta_l \tilde{B}_{T_l,x_{0,l}}(\beta_l x)\] as a weighted linear combination of rescaled and normalized Barenblatt profiles with $\alpha_l \geq 0$, $\sum_{l=1}^m \alpha_l = 1$, $\beta_l>0$ for $l=1,\dots,m$, $m\in \mathbb{N}$. Then, $u_0$ is again a probability density function and due to the composition method, see for example \cite{AsmussenGlynn2007}, we only need a simulation method for the random variables with density $\beta_l \tilde{B}_{T_l,x_{0,l}}(\beta_l x)$. 
To generate these random variables, we use the inverse transformation method. In detail, if $U \sim \mathcal{U}([0,1])$ is a uniformly distributed random variable and \[F(z):= \int_{(-\infty,z]}\tilde{B}_{T,0}(x)dx\] the cumulative distribution function (cdf), then $\xi = F^{-1}(U)$ has the cdf $F$. 
Note that $F^{-1}$ is a right-continuous pseudo-inverse 
\[F^{-1}(v) = \inf\{z \colon F(z) \geq v\}\]
for every $v \in [0,1]$.
A computation shows that the cdf for $\tilde{B}_{T,0}$ is given by
\begin{align*}
F(z) = 
\begin{cases}
0 &\text{ for } z < -\sqrt{12T^{\frac{2}{3}}},\\
\frac{\sqrt{3}}{8}\left(\frac{z}{T^{\frac{1}{3}}}-\frac{z^3}{36T}\right)+\frac{1}{2} &\text{ for } z \in [-\sqrt{12T^{\frac{2}{3}}},\sqrt{12T^{\frac{2}{3}}}),\\
1 &\text{ for } z \geq \sqrt{12T^{\frac{2}{3}}}.
\end{cases}
\end{align*}
To determine the pseudo-inverse of $F$, we let $v \in (0,1)$ and consider the cubic equation 
\begin{align*}
\frac{\sqrt{3}}{8}\left(\frac{z}{T^{\frac{1}{3}}}-\frac{z^3}{36T}\right)+\frac{1}{2} = v \Leftrightarrow& \,z^3+pz+q= 0
\end{align*}
with 
\begin{align*}
p = -36T^{\frac{2}{3}} \text{ and } q = \left(v-\frac{1}{2}\right)96\sqrt{3}T.
\end{align*}
The discriminant of the equation is
\begin{align*}
\Delta &= \left(\frac{q}{2}\right)^2+\left(\frac{p}{3}\right)^3 = T^2 12^3\left(4\left(v-\frac{1}{2}\right)^2-1\right)<0
\end{align*}
since $(v-\frac{1}{2})^2 \in [0,\frac{1}{4})$ and we get three real-valued solutions using Cardano's method. 
From the shape of the function we know that we need the second root 
\[z = -\sqrt{-\frac{4}{3}p}\cos\left(\frac{1}{3} \arccos\left(-\frac{q}{2}\sqrt{-\frac{27}{p^3}}\right)+\frac{\pi}{3}\right).\]
Consequently, we have
\begin{align*}
F^{-1}(v) = 
\begin{cases}
-\infty &\text{ for } v = 0,\\
-\sqrt{-\frac{4}{3}p}\cos\left(\frac{1}{3} \arccos\left(-\frac{q}{2}\sqrt{-\frac{27}{p^3}}\right)+\frac{\pi}{3}\right) &\text{ for } v \in (0,1),\\
\sqrt{12 T^{\frac{2}{3}}} &\text{ for } v =1.
\end{cases}
\end{align*}

By defining $F_l^{-1}(v) = \frac{F^{-1}(v)}{\beta_l}+x_{0,l}$, we obtain $F_l^{-1}(U) \sim \tilde{B}_{T_l,x_{0,l}}(\beta_l x)$ and the complete simulation algorithm for the initial random variables reads:
\begin{enumerate}
\item Generate a random number $I\sim \sum_{l=1}^m \alpha_l \delta_l$
\item Generate a random number $U \sim \mathcal{U}([0,1])$ and use $\xi_i = F_I^{-1}(U)\sim u_0$
\end{enumerate}


\subsection{Discretization of the diffusion-aggregation equation} \label{subsec:NumMac}
Next, we derive a numerical scheme for the macroscopic equation \eqref{eqn} 
restricted to one space dimension here. 
The latter equation is a positivity-preserving conservation law and from Theorem \ref{existWeak} we know that there exists a global solution if $\eta:=\frac{a}{2b ||u_0||_{L^\infty}} >1$. 
We rewrite the equation as follows
\begin{eqnarray*}
\partial_t u+2 b(u_x u)_x = a u_{xx}.
\end{eqnarray*}
From the assumptions $a,b\geq 0$, we can identify
the classical heat equation $\partial_tu = a u_{xx}$  
and an advection equation $\partial_t u +2b(u_x u)_x = 0,$
where the flux also depends on the derivative of the solution, see e.g.\ \cite{LevequeRed}. 
The reinterpretation of a nonlinear heat equation as conservation law has been studied for degenerated parabolic partial differential equations in \cite{Bessemoulin2012,Buerger2008, Liu2011, Pandian1989}.

We use a fractional step method \cite{LevequeRed} to separately solve the classical linear diffusion and the advection part in one time step. 
The classical linear diffusion is aprroximated by the explicit first order method
$\mathcal{H}^{(1)}\colon \mathbb{R}^\mathbb{Z} \to \mathbb{R}^\mathbb{Z}$ with
\begin{align}
(\mathcal{H}^{(1)}(u))_i &= u_i+a \Delta t D^+D^-(u)_i, \label{eq:HeatEq}
\end{align}
where \begin{align*}
D^+D^- (u)_i = \frac{u_{i+1}-2u_i+u_{i-1}}{(\Delta x)^2}
\end{align*}
is the finite difference approximation of the second derivative. 
From literature we know that
this linear method is $||\cdot||_\infty$-stable, i.e.\ $||\mathcal{H}^{(1)}(u)||_\infty \leq ||u||_\infty$ if 
\begin{equation}\label{eq:stab}
a\frac{\Delta t}{(\Delta x)^2}\leq \frac{1}{2}.
\end{equation}
In a second step, we approximate the advection part by an adapted upwind-scheme $\mathcal{H}^{(2)}\colon \mathbb{R}^\mathbb{Z} \to \mathbb{R}^\mathbb{Z}$ in conservative form 
\begin{align}
(\mathcal{H}^{(2)}(u))_i &= u_i-\frac{\Delta t}{\Delta x} (\mathcal{F}_{i+\frac{1}{2}}(u)-\mathcal{F}_{i-\frac{1}{2}}(u)) \label{eq:AdvEq}
\end{align}
with numerical fluxes $F_{i-\frac{1}{2}}(u)$. Since the flux function depends on the derivative of the solution, we first approximate the first derivative with the central difference \[D^0(u)_i := \frac{u_{i+1}-u_{i-1}}{2 \Delta x}\]
and set \[\partial_{i+\frac{1}{2}}(u) := \frac{D^0(u)_{i+1}+D^0(u)_i}{2}\]
as the approximation of $u_x$ at the right boundary of the cell $[x_{i-\frac{1}{2}},x_{i+\frac{1}{2}}]$. The numerical flux is then defined by
\begin{align*}
\mathcal{F}_{i+\frac{1}{2}}(u) := 
\begin{cases}
2b \partial_{i+\frac{1}{2}}(u) u_i \quad &\text{ if } \quad \partial_{i+\frac{1}{2}}(u) \geq 0, \\
2b \partial_{i+\frac{1}{2}}(u) u_{i+1} \quad &\text{ if } \quad \partial_{i+\frac{1}{2}}(u) < 0.
\end{cases}
\end{align*}
To ensure that the analytic domain of dependence is included in the numerical domain of dependence, the CFL condition 
\begin{align}
\frac{\Delta t}{\Delta x} \max_{i \in \mathbb{Z}}\{|2b \partial_{i+\frac{1}{2}}(u)|\} \leq 1 \label{eq:CFL}
\end{align}
must be satisfied in each iteration-step.
Fusing both discretization approaches for the diffusion and advection part leads to the numerical scheme $\mathcal{H} \colon \mathbb{R}^\mathbb{Z} \to \mathbb{R}^\mathbb{Z}$ defined by $\mathcal{H} := \mathcal{H}^{(1)} \circ \mathcal{H}^{(2)}$.

Provided the initial data is positive, the solution of the diffusion-aggregation equation remains positive, see Theorem \ref{existWeak}. 
This property shall be also hold for the numerical scheme 
and thus we must guarantee that $\mathcal{H}$ is positive-preserving, cf. \cite{ChertockKurganov2008,ChertockKurganovWang2012} for the chemotaxis model. 
In contrast to a straightforward approximation of the diffusion-aggregation equation, the numerical scheme we propose is positive-preserving independent of the choice of the parameters $a,b \geq 0$, see Theorem \ref{thm:pos_pres}. This allows to numerically evaluate the transition from the diffusion to the aggregation regime in section \ref{sec:num_res}.
\begin{thm}\label{thm:pos_pres}
The numerical scheme $\mathcal{H}$ is positive-preserving if 
\begin{align}
\Delta t \leq \min\left\{\frac{(\Delta x)^2}{2a},\frac{\Delta x}{4b \max_{i \in \mathbb{Z}}\{|\partial_{i+\frac{1}{2}}(u)|\}}\right\}, \label{eq:CFL2}
\end{align}
i.e.\ for all $u \in \mathbb{R}^\mathbb{Z}$ with $u \geq 0$ it follows $\mathcal{H}(u) \geq 0$.
\end{thm}
\begin{proof}
If both methods $\mathcal{H}^{(1)}$ and $\mathcal{H}^{(2)}$ are positive-preserving, then the composition $\mathcal{H}$ is also positive preserving. Let $u \in \mathbb{R}^\mathbb{Z}$ with $u \geq 0$ be given. We have
\begin{align*}
\mathcal{H}^{(1)}(u)_i &= u_i+a \frac{\Delta t}{(\Delta x)^2}(u_{i+1}-2u_i+u_{i-1})\\
	&= u_i(1-2a\frac{\Delta t}{(\Delta x)^2})+a\frac{\Delta t}{(\Delta x)^2}(u_{i+1}+u_{i-1})\\
	&\geq 0
\end{align*}
due to condition \eqref{eq:stab} and $a\geq 0$. This shows $\mathcal{H}^{(1)}(u) \geq 0$. 
To prove that $\mathcal{H}^{(2)}(u) \geq 0$, we consider the following estimate
\begin{align*}
	\mathcal{H}^{(2)}(u)_i &= u_i-2b\frac{\Delta t}{\Delta x}(\partial_{i+\frac{1}{2}}(u) u_i- \partial_{i-\frac{1}{2}}(u)u_{i})\\
				&= u_i\big(1-2b \frac{\Delta t}{\Delta x} (|\partial_{i+\frac{1}{2}}(u)|+|\partial_{i-\frac{1}{2}}(u)|)\big).
\end{align*}
From \eqref{eq:CFL2} we know that
\begin{align*}
2b \frac{\Delta t}{\Delta x} \big(|\partial_{i+\frac{1}{2}}(u)|+|\partial_{i-\frac{1}{2}}(u)|\big) \leq 1
\end{align*}
and consequently $\mathcal{H}^{(2)}(u)_i  \geq 0$.
\end{proof}
In the next section, we present numerical results based on the discretizations proposed for the stochastic particle system
and the diffusion-aggregation equation.

\section{Numerical results}\label{sec:num_res}

\subsection{Results for the stochastic particle system}
We consider the Euler-Maruyama scheme \eqref{eq:Euler1}-\eqref{eq:Euler2} for the stochastic particle system and fix $b = 1$ in the following and vary $a$ according to the relation $a = 2b||u_0||_{L^\infty}\eta$ for $\eta \geq 0$. Theorem \ref{existWeak} motivates to distinguish two cases
for the diffusion-aggregation equation: 
we call $\eta>1$ the diffusion and $\eta<1$ the aggregation case. 
We choose the initial densities as the superposition of normalized Barenblatt profiles introduced in section \ref{subsec:ParticleSystem}. In detail, we use $m = 3$ profiles with 
$$x_0 = \big(-\sqrt{12 T_1^\frac{2}{3}}(1+\frac{1}{\beta_1}),0,\sqrt{12 T_3^\frac{2}{3}}(1+\frac{1}{\beta_3})\big)$$ 
such that the corresponding supports are disjoint and set $T_l = 2$ for $l = 1,2,3$. The two different initial distributions $u_0$ we consider in the following are given in figure \ref{fig:PDFInitial}.

Furthermore, the interaction kernel $V$ is chosen as the density of a standard normal distribution with weight $b$, i.e. 
$
V(x) = \tfrac{b}{\sqrt{2\pi}}e^{-\tfrac{x^2}{2}}.
$

Note that the parameters $\epsilon$ and $N$ significantly influence the results of the stochastic particle system and need to be chosen in an appropriate way to provide results which are close to the diffusion-aggregation equation. 
We use the estimate \eqref{eq:Estimate} which states an upper bound on the squared $L^2-$distance between the particle and intermediate model.
%
We choose the time horizon $T = 7$, the parameter $\epsilon = 1.5$ and determine $||V^{\prime \prime}||_\infty = \tfrac{b}{\sqrt{2\pi}}$ for the interaction kernel $V$. Then, a particle number of $N = 555$ ensures that the squared $L^2-$distance between the stochastic particle and intermediate model is smaller than $0.3 \cdot C$, where $C$ is the constant in \eqref{eq:Estimate}.

In figures \ref{fig:ResMicPDF1} and \ref{fig:ResMicPDF2} the mean density of 1000 Monte-Carlo samples for both initial distributions $u_0$ is shown. In the cases $\eta = 0.1$ and $\eta = 0.2,$ we can observe the aggregation at local maxima, whereas in the cases $\eta = 0.6$ and $\eta = 1$ 
we observe a more diffusive behavior. 
This is emphasized by figure \ref{fig:ResMicSup}, where the running supremum $t \mapsto \sup\{||u(s,\cdot)||_\infty \colon s\leq t\}$ is drawn. 

For this choice of parameters,
we expect the threshold between aggregation and diffusion to be between $\eta = 0.2$ and $\eta = 0.6$ . 
The maximal value of the sampled mean density behaves almost linear in time and is directly related to the value of $\eta$. Concerning the different initial values, there is no severe difference between the shapes of the running supremum.

Since we are not in the limiting regime $\epsilon \to 0$ and $N \to \infty$, we already observe diffusion for $\eta \ll 1$. In detail, the diffusion in the mean density arising from the Brownian motion is independent of the number of particles, whereas the aggregation highly depends on the number of particles $N$ and the range of strong interactions measured by $\epsilon$.

\begin{figure}[H]
\centering
\subfloat{
\includegraphics[width=0.49\textwidth]{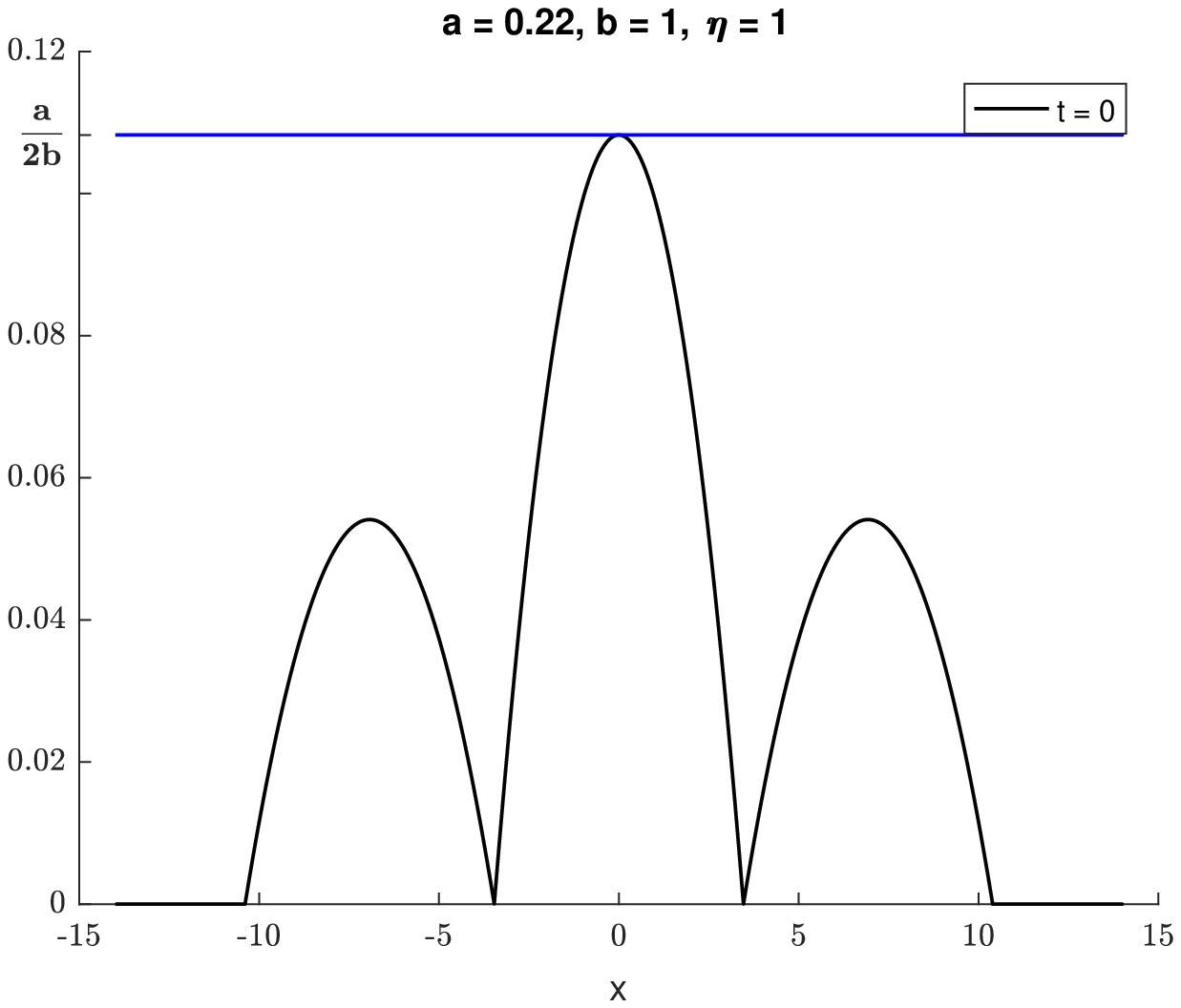}
}
\subfloat{
\includegraphics[width=0.49\textwidth]{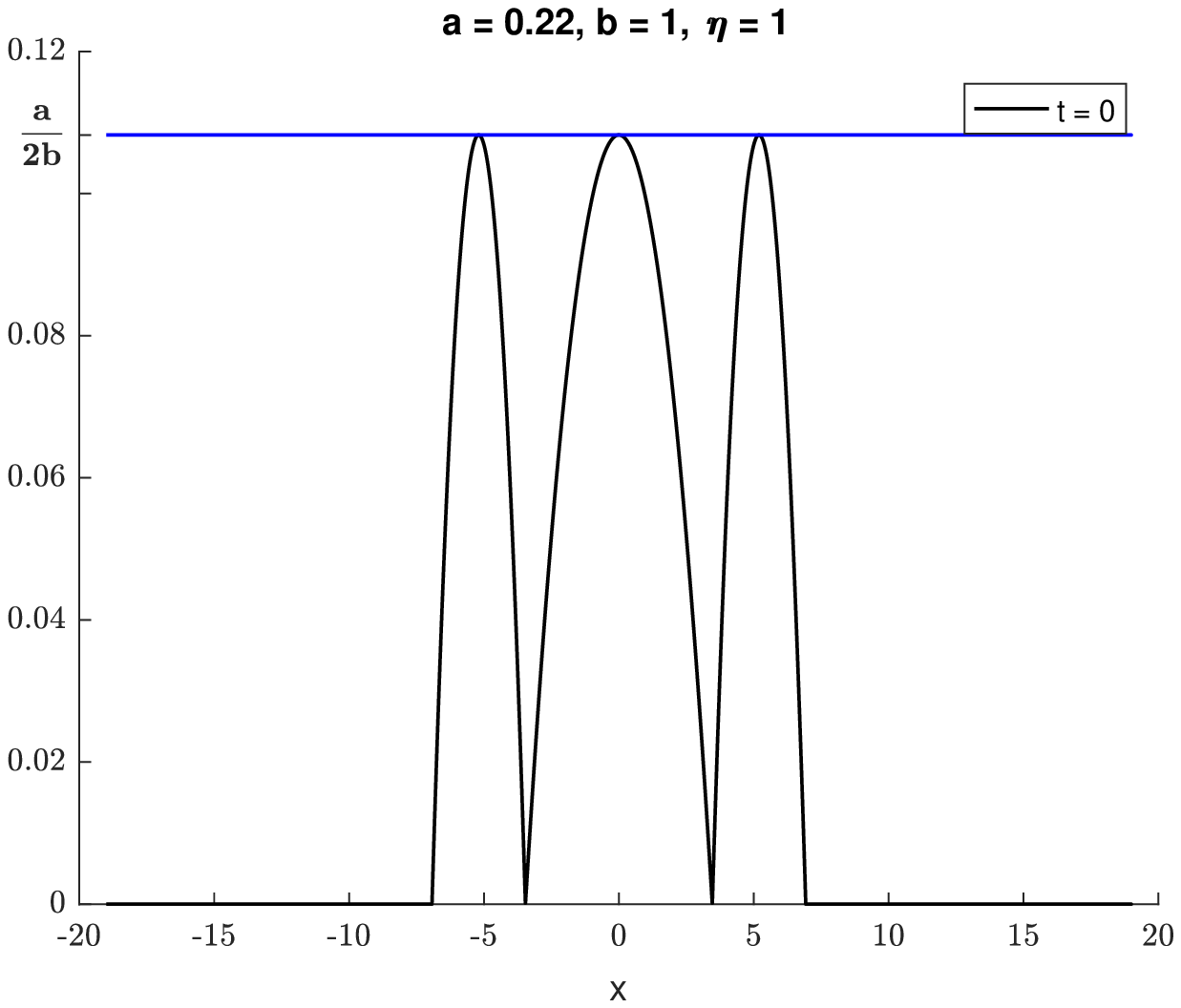}
}
\caption{Pdf's of the initial distribution with parameters from left to the right: $\alpha = (\frac{1}{4},\frac{1}{2},\frac{1}{4}), \beta = (1,1,1)$; $\alpha = (\frac{1}{4},\frac{1}{2},\frac{1}{4}), \beta = (2,1,2)$}
\label{fig:PDFInitial}
\end{figure} \vspace{-5mm}
\begin{figure}[H]
\centering
\subfloat{
\includegraphics[width=0.49\textwidth]{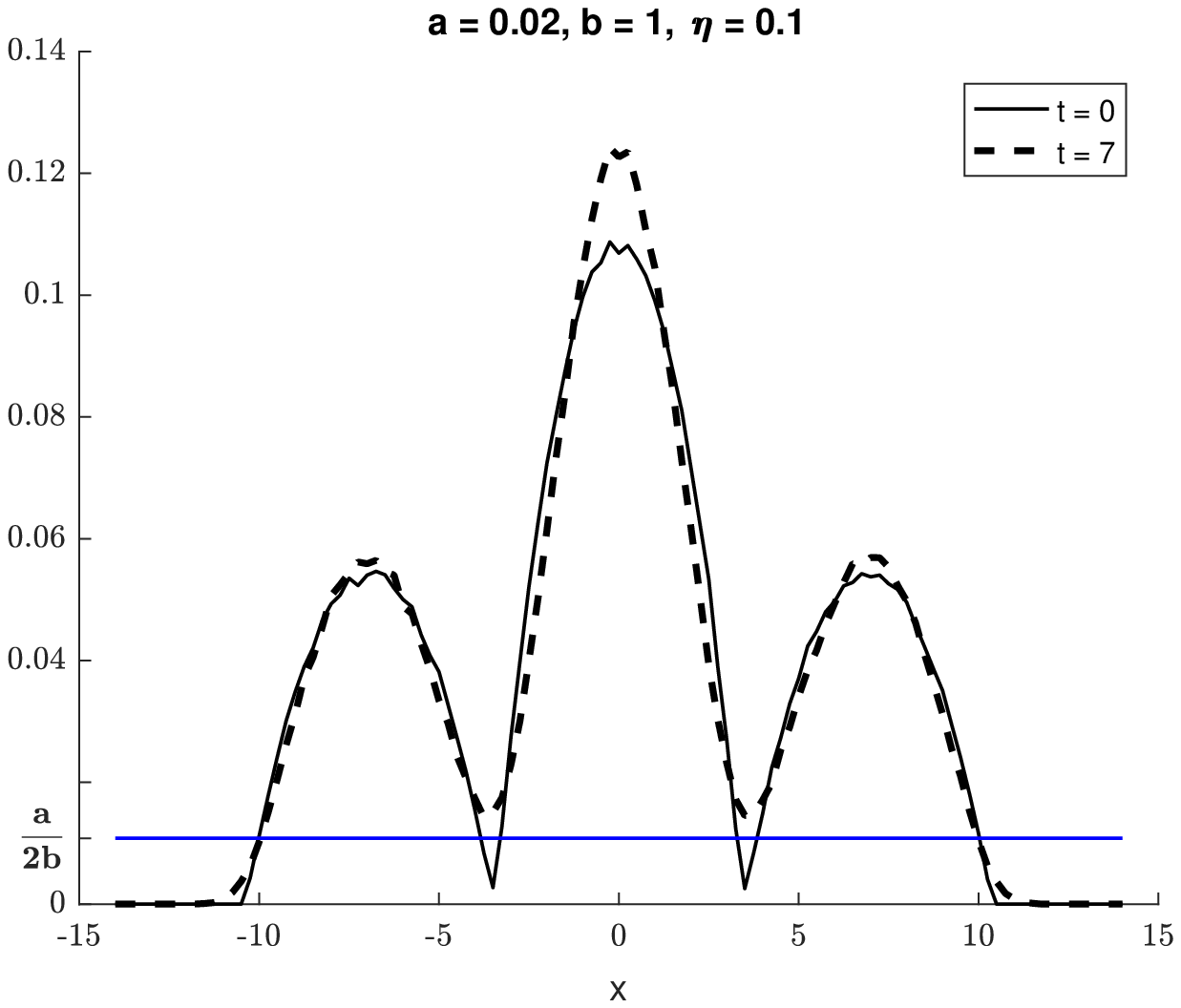}
}
\subfloat{
\includegraphics[width=0.49\textwidth]{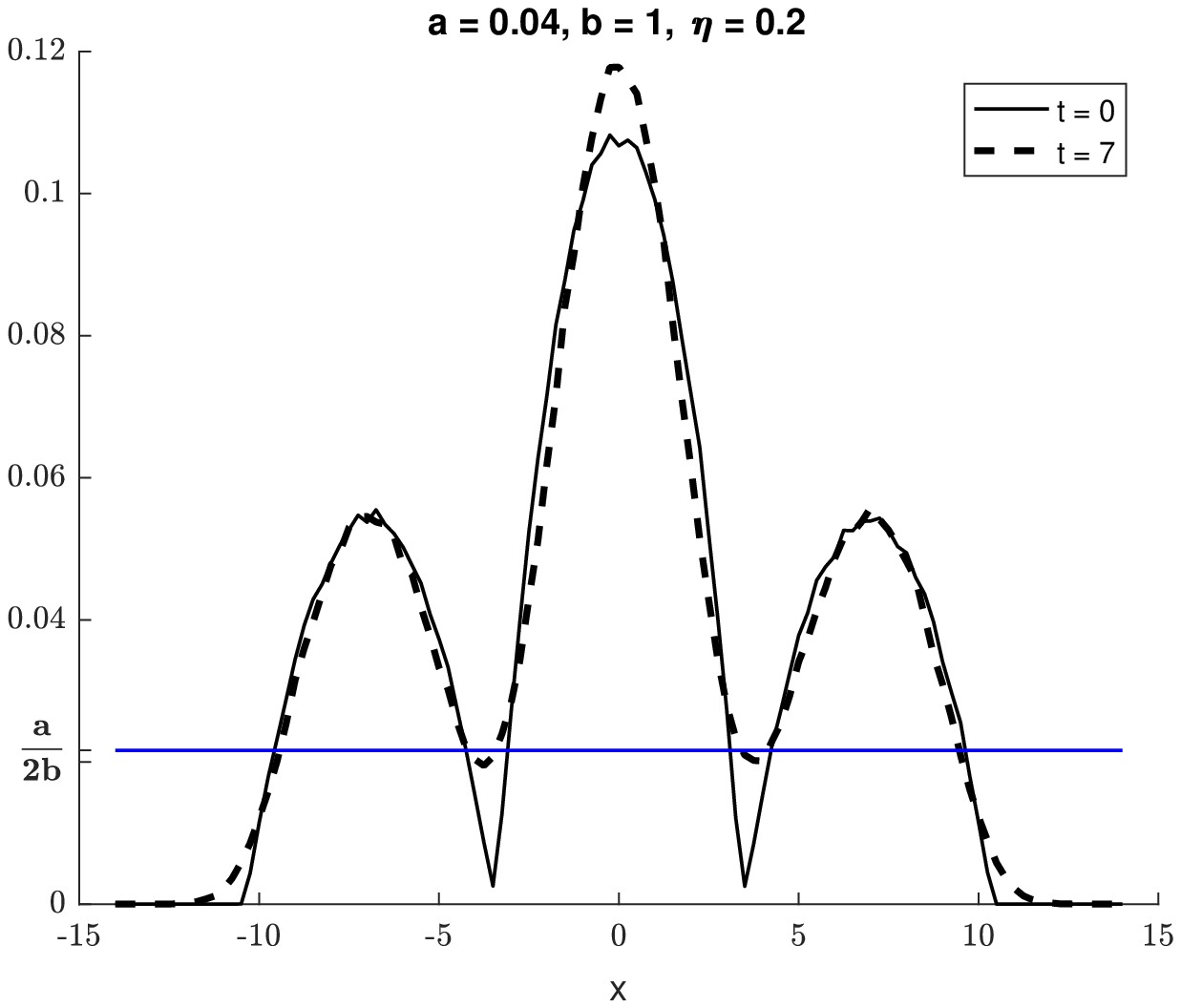}
}\\
\subfloat{
\includegraphics[width=0.49\textwidth]{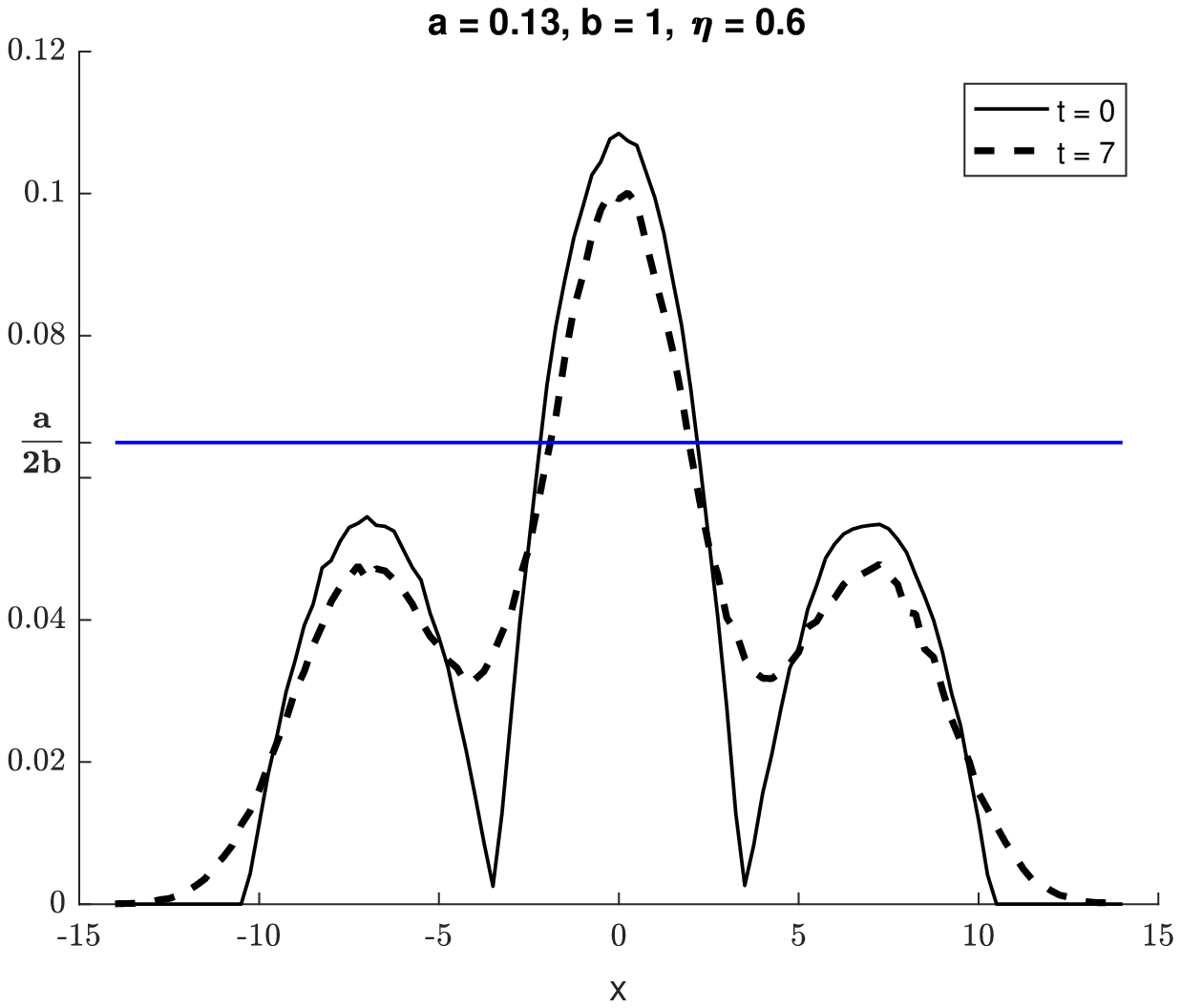}
}
\subfloat{
\includegraphics[width=0.49\textwidth]{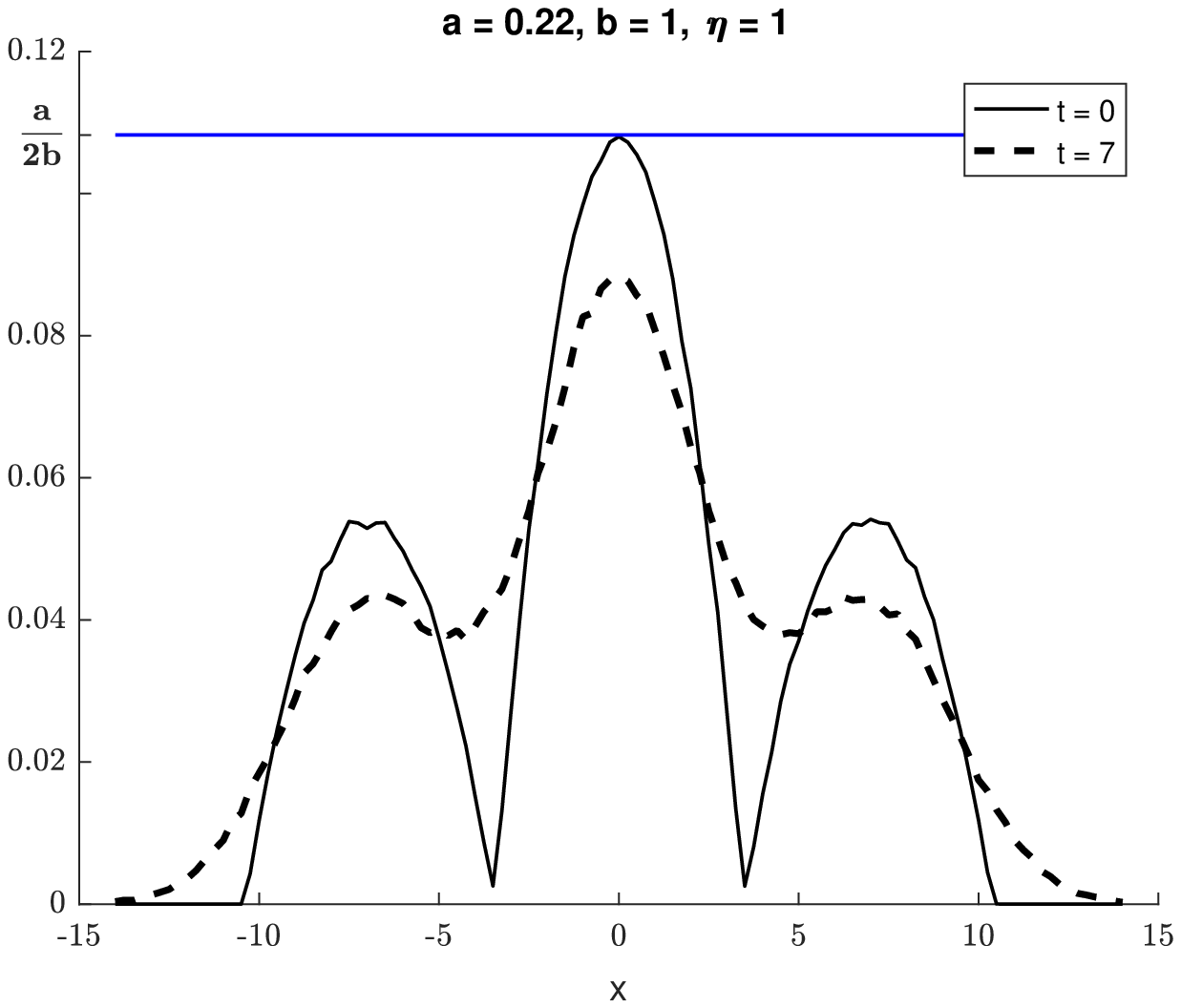}
}
\caption{Results for the stochastic particle system with $N=555$ particles, initial distribution 1 and different values for $\eta$}
\label{fig:ResMicPDF1}
\end{figure}
\begin{figure}[H]
\centering
\subfloat{
\includegraphics[width=0.49\textwidth]{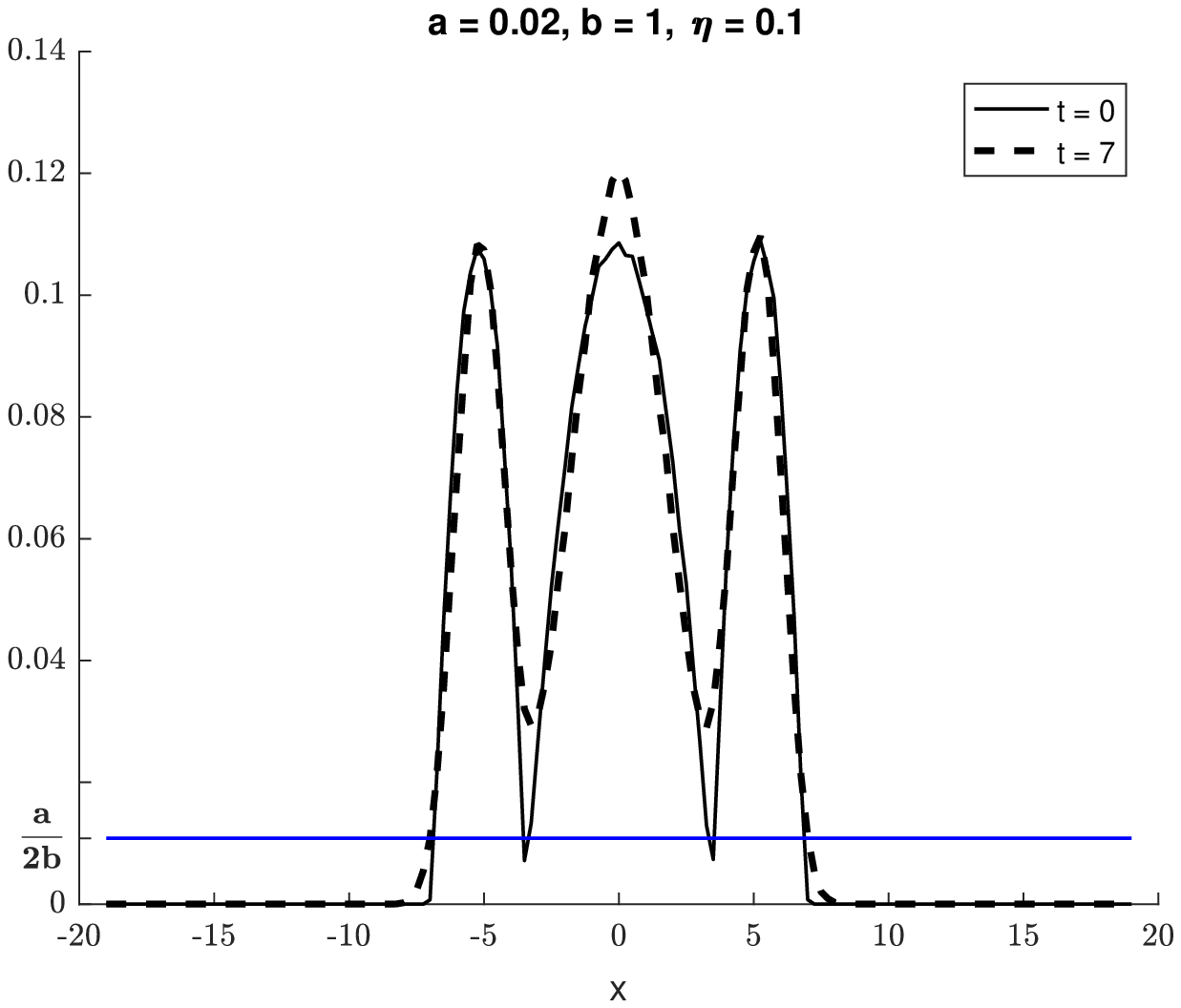}
}
\subfloat{
\includegraphics[width=0.49\textwidth]{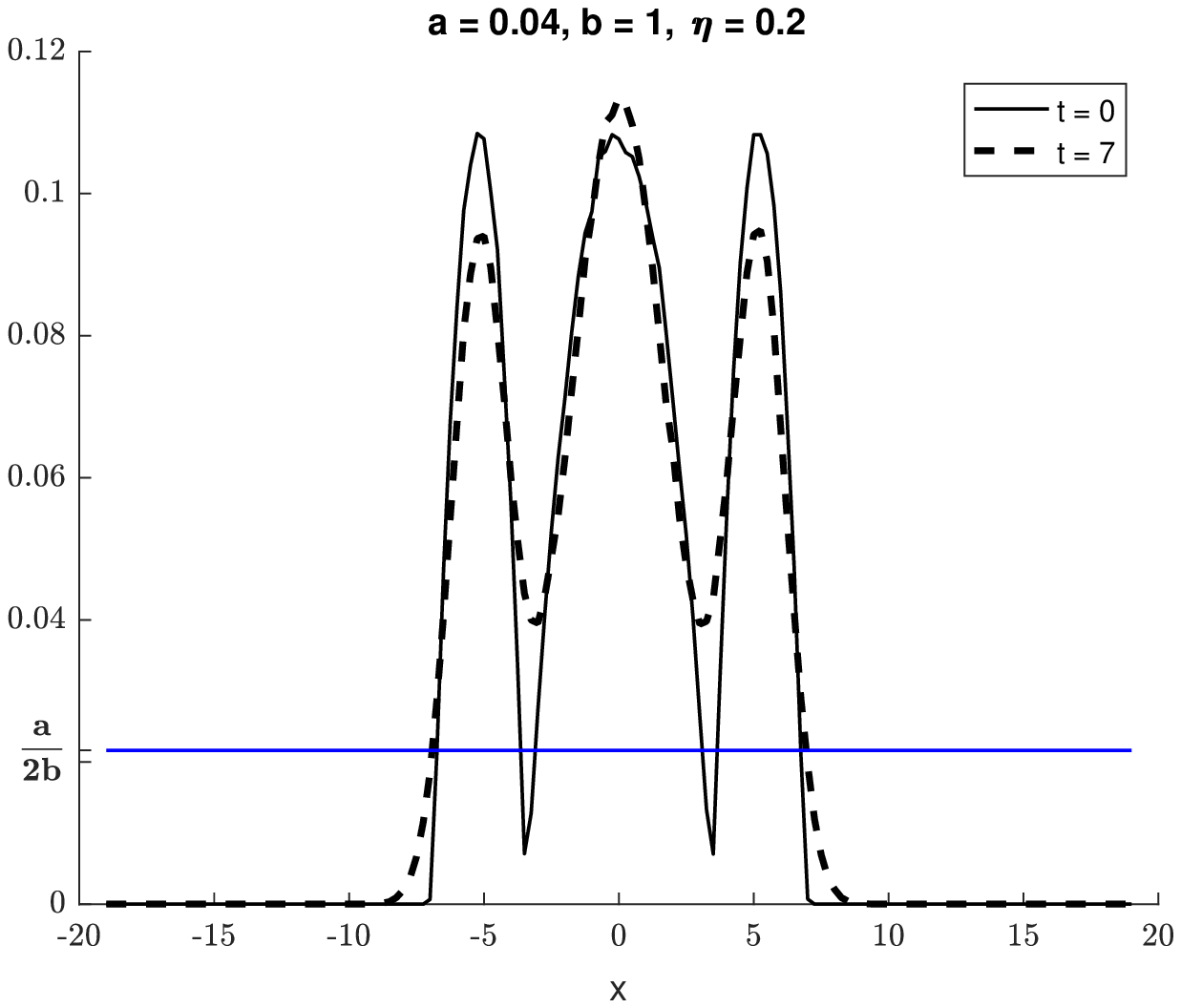}
}\\
\subfloat{
\includegraphics[width=0.49\textwidth]{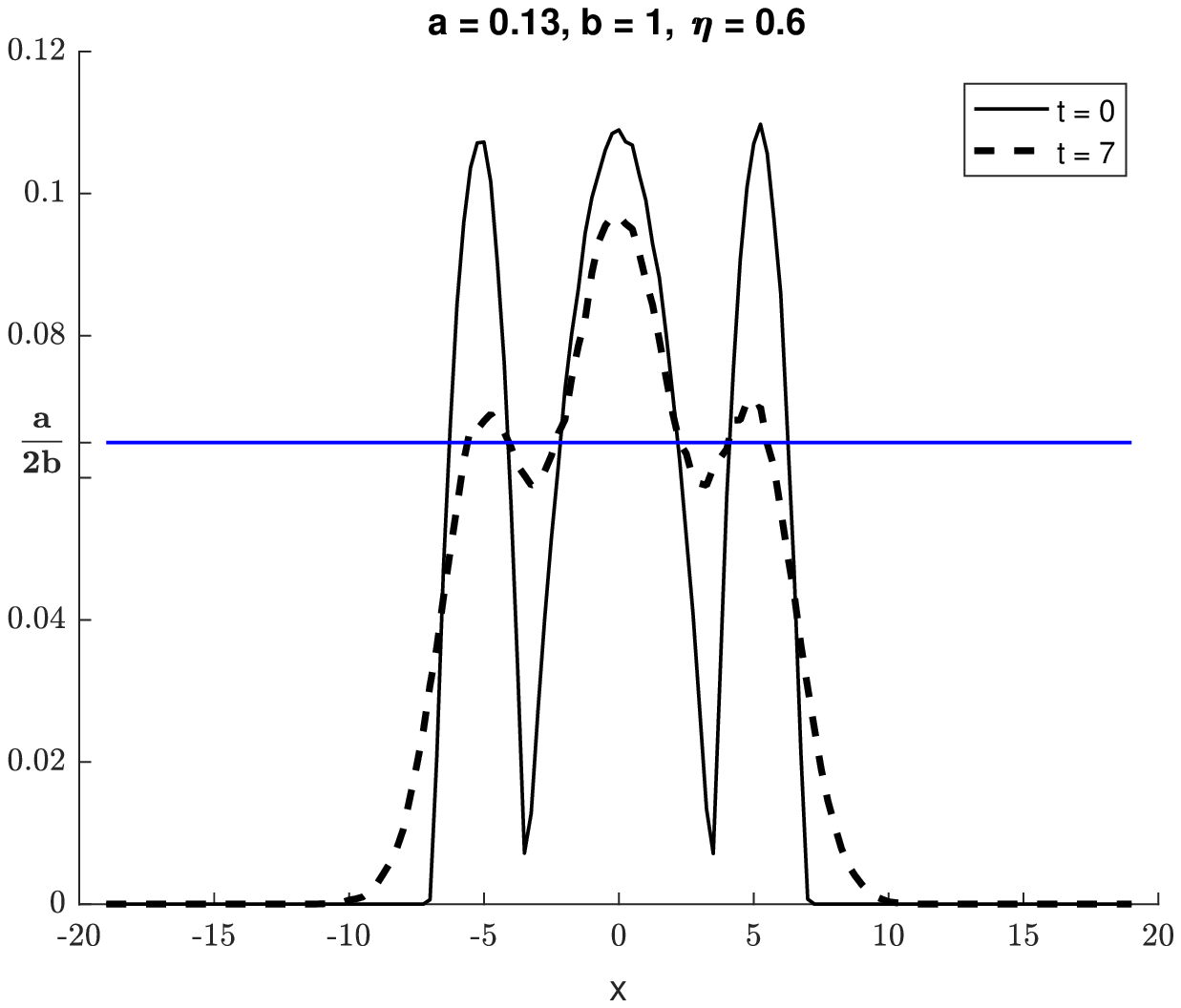}
}
\subfloat{
\includegraphics[width=0.49\textwidth]{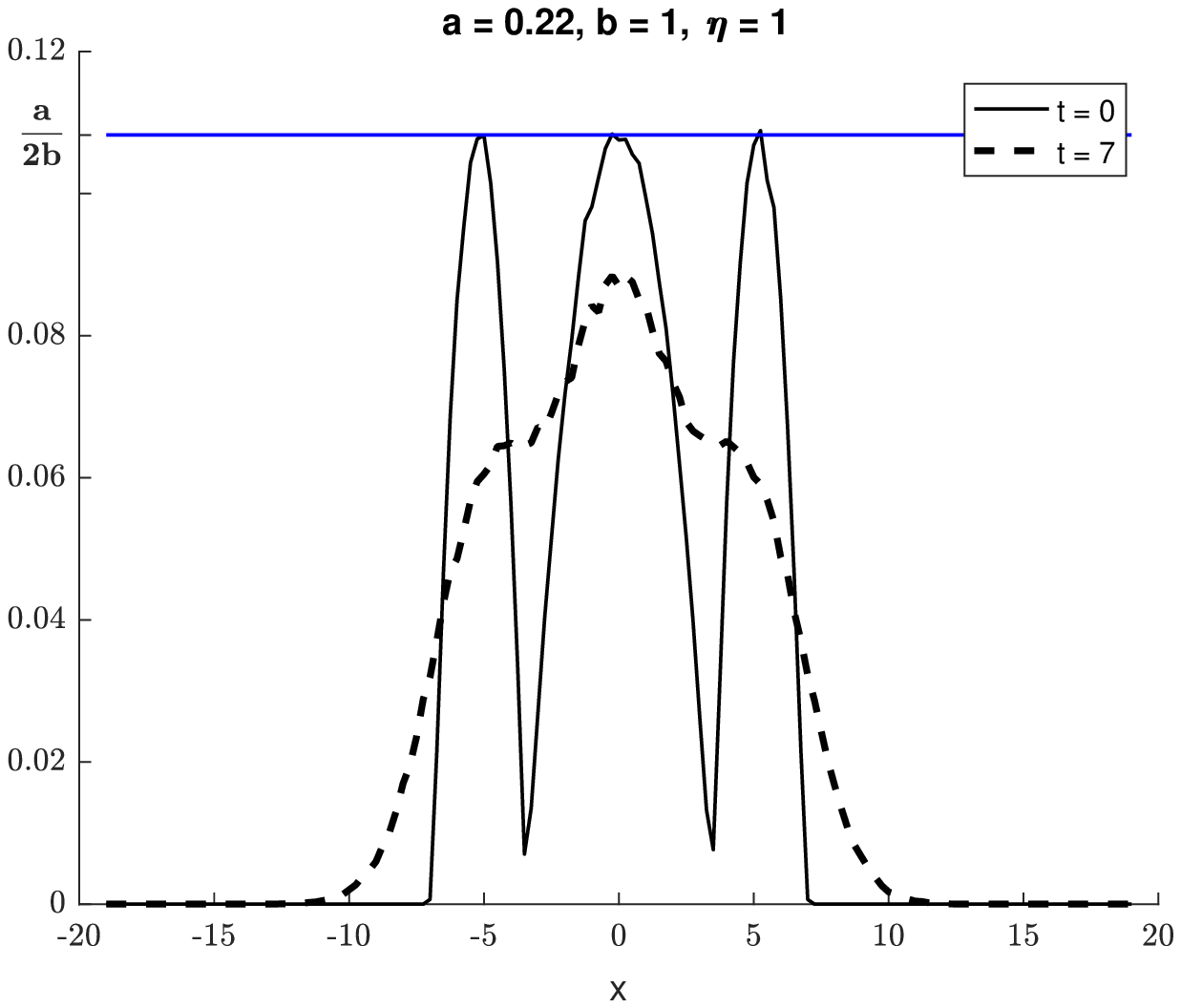}
}
\caption{Results for the stochastic particle system with $N=555$ particles, initial distribution 2 and different values for $\eta$}
\label{fig:ResMicPDF2}
\end{figure}\vspace{-5mm}
\begin{figure}[H]
\centering
\subfloat{
\includegraphics[width=0.49\textwidth]{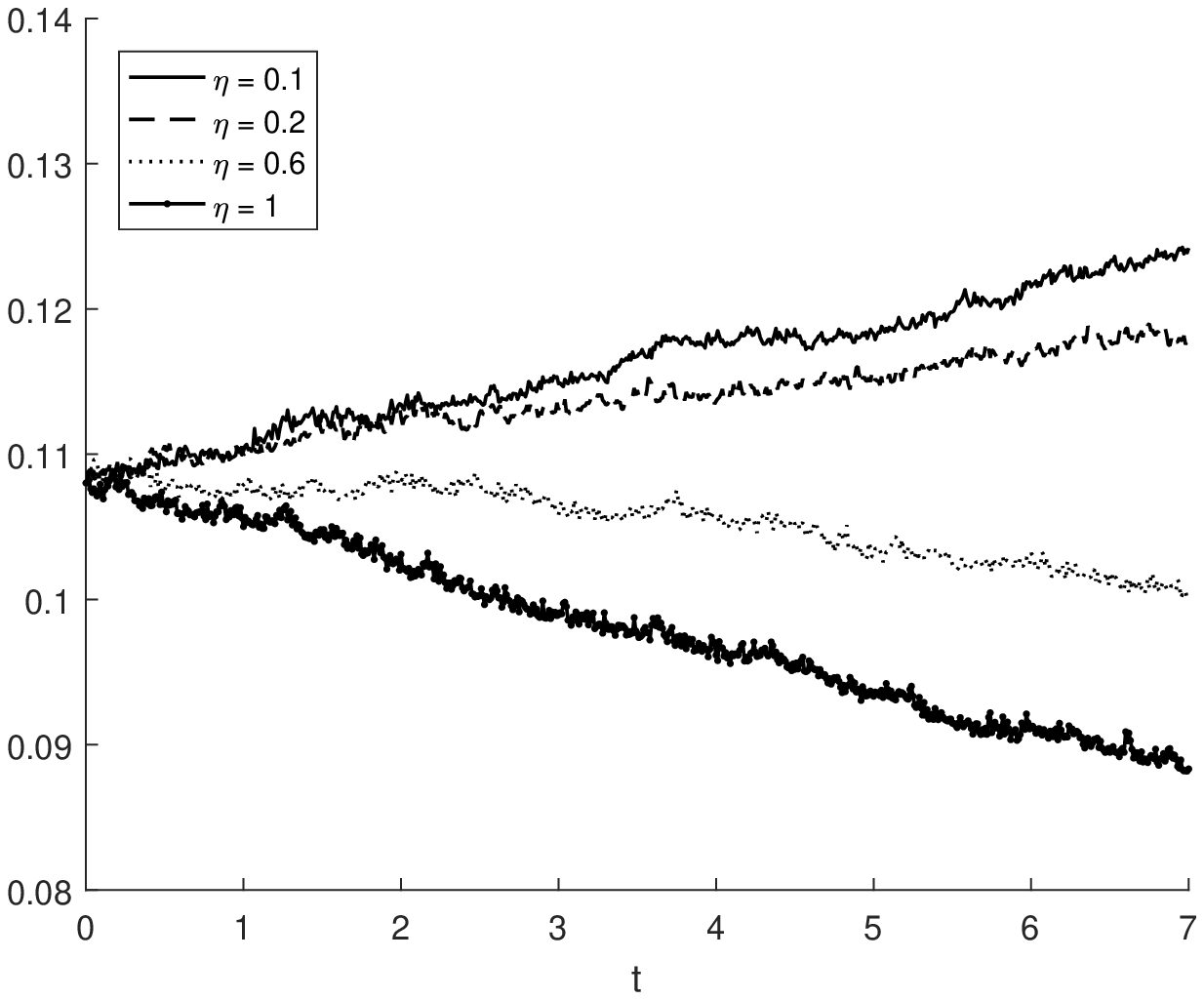}
}
\subfloat{
\includegraphics[width=0.49\textwidth]{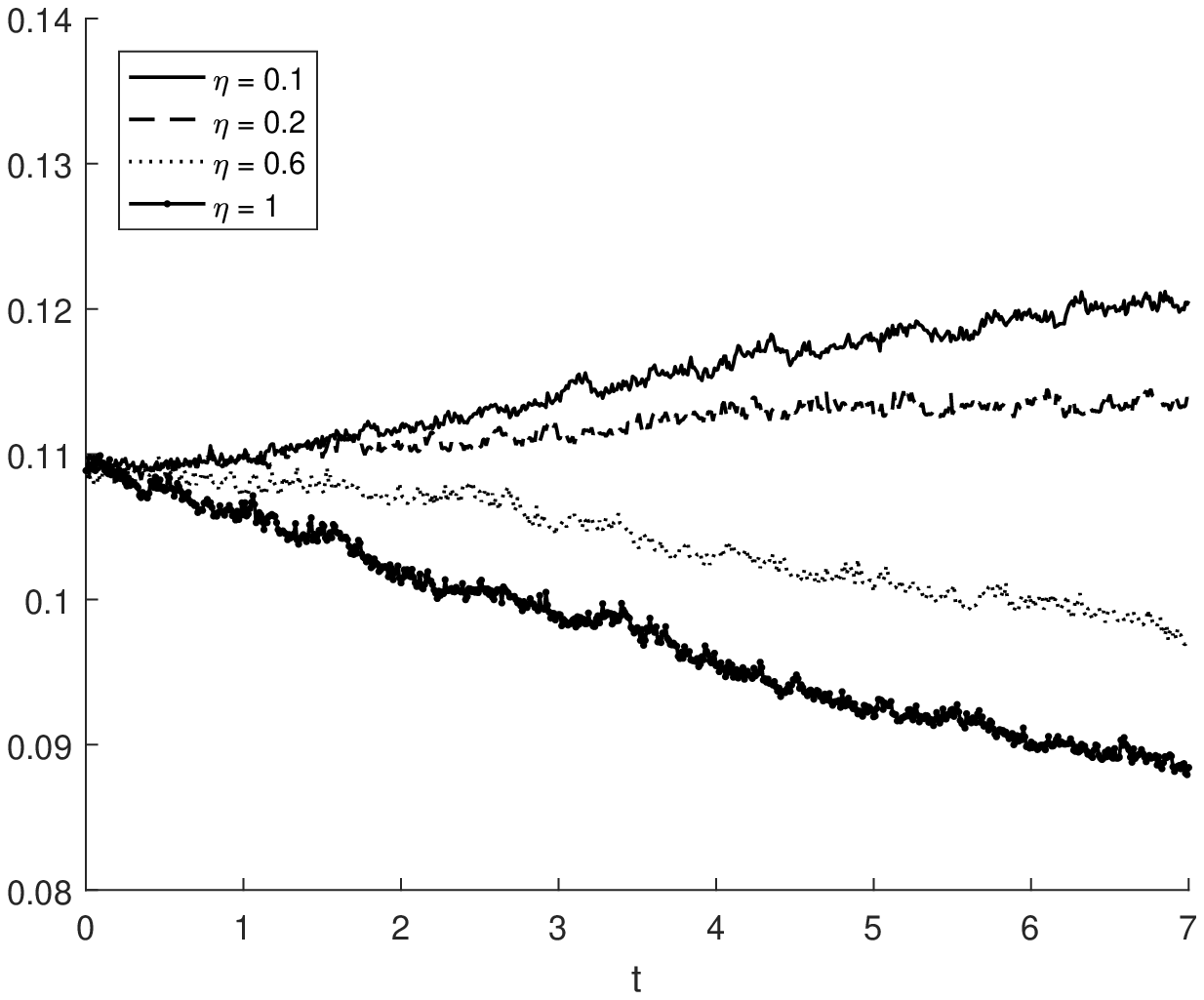}
}
\caption{Running supremum of the stochastic particle density with $N=555$ particles, initial distribution 1 (left) and 2 (right) and different values for $\eta$}
\label{fig:ResMicSup}
\end{figure}

\subsection{Results for the diffusion-aggregation equation}
Now, we apply the numerical scheme from subsection \ref{subsec:NumMac} to both initial profiles, see figure \ref{fig:PDFInitial}. 
In particular, we aim to analyze the threshold $\eta = 1$ which is not covered by the theoretical results.
A simulation result for this choice of $\eta$ can be found in figure \ref{fig:ResMacDiffusion} and 
indicates a diffusive behavior for each initial data. 
The results rely on a fine-scale resolution with 
spatial step-size $\Delta x = 2^{-8}$ and time step-size according to \eqref{eq:CFL2}.
\begin{figure}[H]
\centering
\subfloat{
\includegraphics[width=0.49\textwidth]{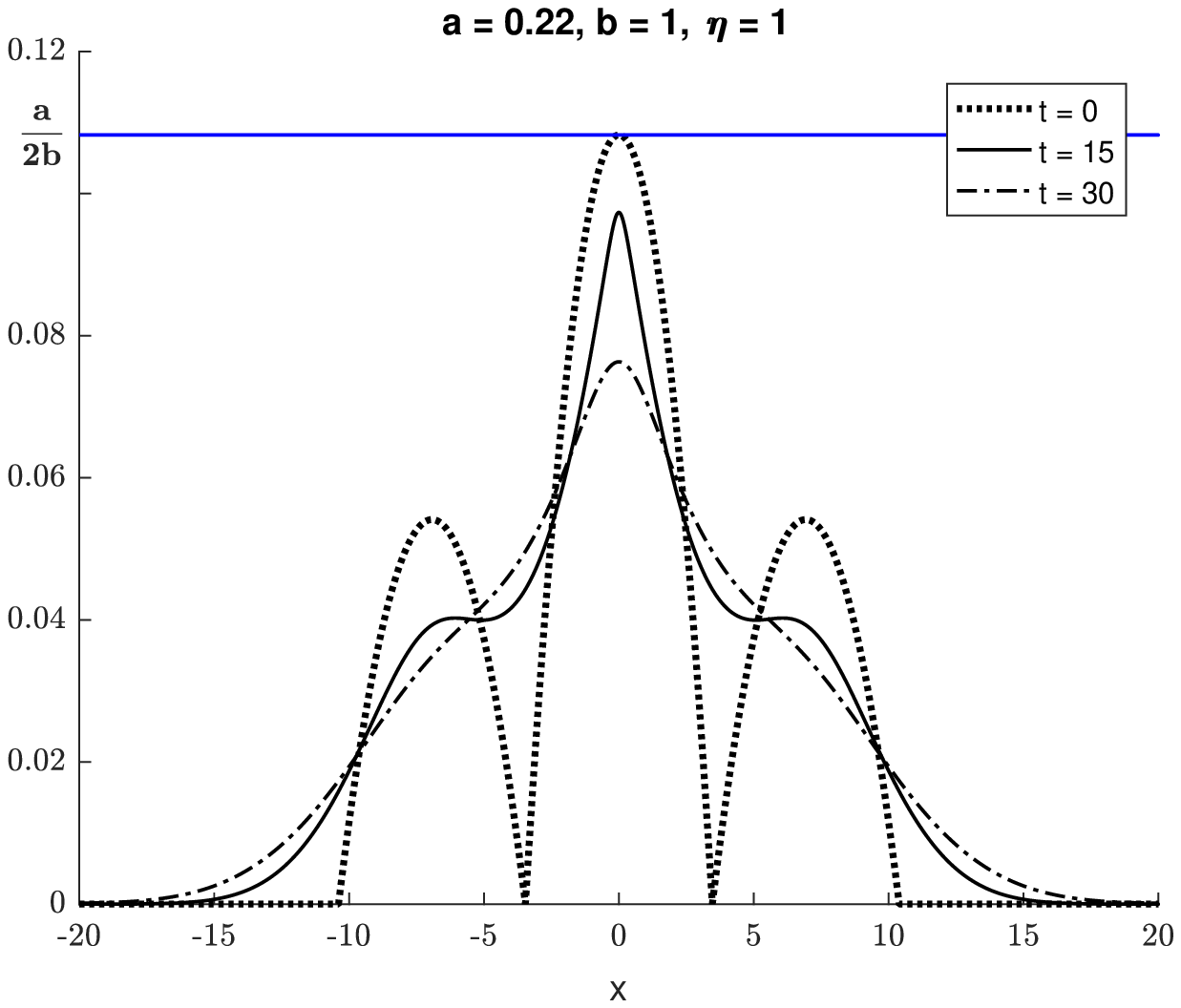}
}
\subfloat{
\includegraphics[width=0.49\textwidth]{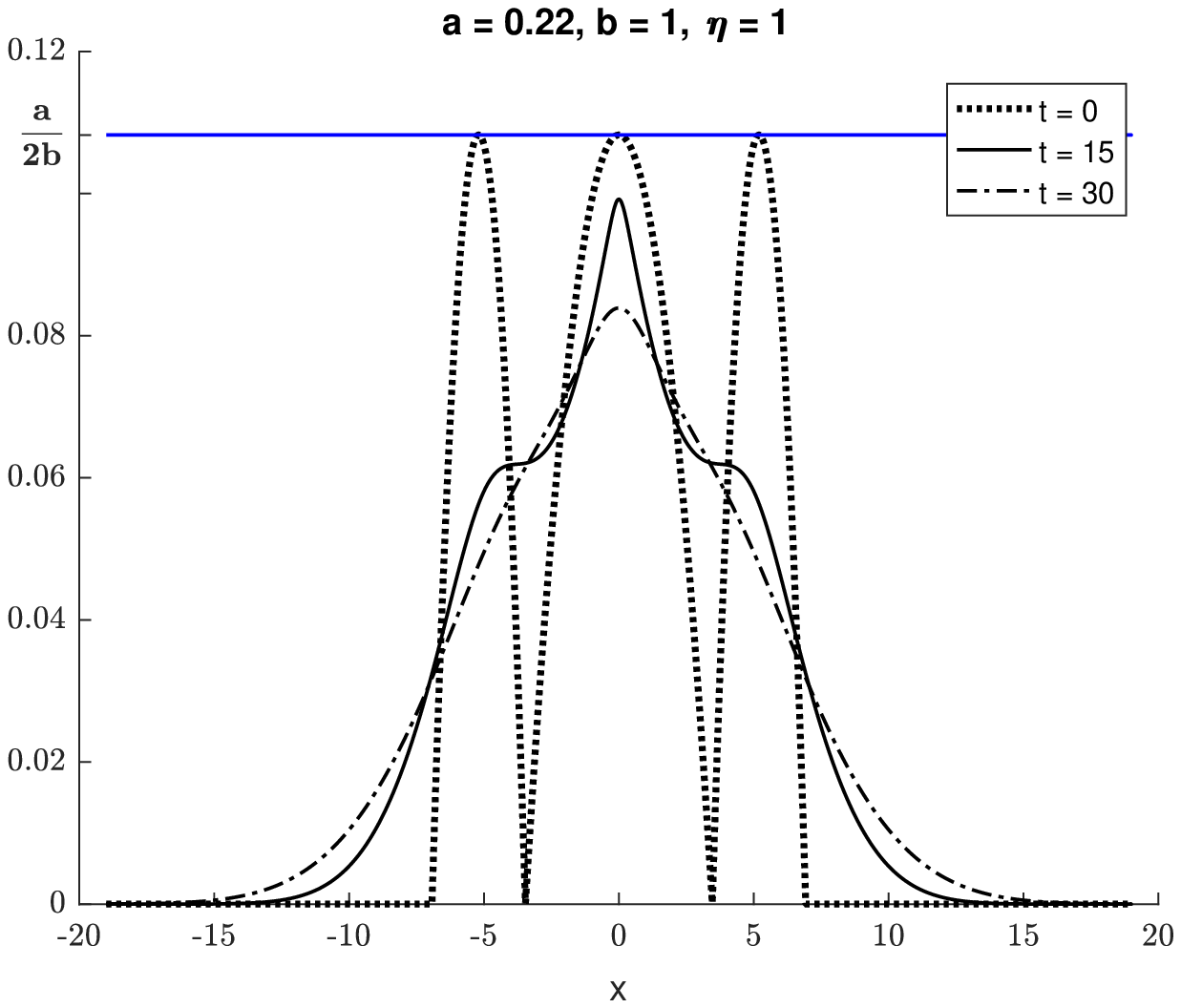}
}
\caption{Mean density for initial data 1 (left) and 2 (right)}
\label{fig:ResMacDiffusion}
\end{figure}

To experimentally verify the convergence of the numerical scheme, we take the reference solution computed with step-size $\Delta x = 2^{-8}$ and study the experimental order of convergence (EOC) with the step-sizes $\Delta x = 2^{-\iota}, \iota = 1,\dots,7$. 
We choose the discrete $L^1-$error \[\texttt{err} = \max_{j}\Delta x \sum_{i \in \mathbb{Z}} |u^{\text{ref}}_i(t_j)-u_i(t_j)|\]  
to measure the difference.

The second and third column of table \ref{tab:convMacro} contain the error and EOC for case 1 (left picture in \ref{fig:ResMacDiffusion}) and 
columns four and five the results for case 2 (right picture in \ref{fig:ResMacDiffusion}), respectively. In both cases, the EOC seems to be at least of order one and the numerical scheme appears to work well regarding the $L^1-$error.
\begin{table}[H]
\centering
\begin{tabular}{c|c|c|c|c}
step-size&\texttt{err} case 1&EOC case 1&\texttt{err} case 2&EOC case 2\\ \hline
$\Delta x = 2^{-1}$&$23.745 \cdot 10^{-3}$&$-$&$48.652 \cdot 10^{-3}$&$-$ \\ \hline
$\Delta x = 2^{-2}$&$12.298 \cdot 10^{-3}$&$0.949$&$26.969 \cdot 10^{-3}$&$0.851$ \\ \hline
$\Delta x = 2^{-3}$&$6.235 \cdot 10^{-3}$&$0.980$&$14.211 \cdot 10^{-3}$&$0.924$ \\ \hline
$\Delta x = 2^{-4}$&$3.501 \cdot 10^{-3}$&$0.833$&$7.479 \cdot 10^{-3}$&$0.926$ \\ \hline
$\Delta x = 2^{-5}$&$1.659 \cdot 10^{-3}$&$1.078$&$3.081 \cdot 10^{-3}$&$1.280$ \\ \hline
$\Delta x = 2^{-6}$&$0.741 \cdot 10^{-3}$&$1.162$&$1.211 \cdot 10^{-3}$&$1.346$ \\ \hline
$\Delta x = 2^{-7}$&$0.270 \cdot 10^{-3}$&$1.458$&$0.400 \cdot 10^{-3}$&$1.600$ \\ \hline
\end{tabular}
\caption{Numerical convergence for $\eta = 1$ and $\Delta x = 2^{-8}$ as reference step-size for the initial density 1 and 2}
\label{tab:convMacro}
\end{table}

Due to numerical diffusion arising from the upwind scheme, we cannot expect a strict regime switch at the theoretical threshold $\eta = 1$.
Depending on the mesh-size, the threshold is expected to be lower than 1.
We examine the threshold by the running supremum $t \mapsto \sup\{||u(s,\cdot)||_\infty \colon s\leq t\}$.  
Figure \ref{fig:ResMacRunSup} shows the running supremum for different values of $\eta$ close to 1. 
From the values of $\eta$ and the shape of the corresponding running supremum, we observe a strict distinction of the diffusion and aggregation regime 
as theoretically assumed.
Additionally, if $\eta$ is decreased  a blow-up occurs and conversely, if $\eta$ is increased, 
the diffusion dominates the supremum. 
In the cases, where the solution follows a diffusive behavior, we observe an increasing supremum until the time $t=12$ which occurs at the center of the 
given profiles, see left picture in figure \ref{fig:ResMac}. 
In the case of initial data 2, the approximated solution increases first at the left and right maxima, see right picture in figure \ref{fig:ResMac},
which is due to the higher slope close to the peaks. The effect of first increasing and then decreasing solutions might be the result of numerical diffusion.

\begin{figure}[htb!]
\centering
\subfloat{
\includegraphics[width=0.49\textwidth]{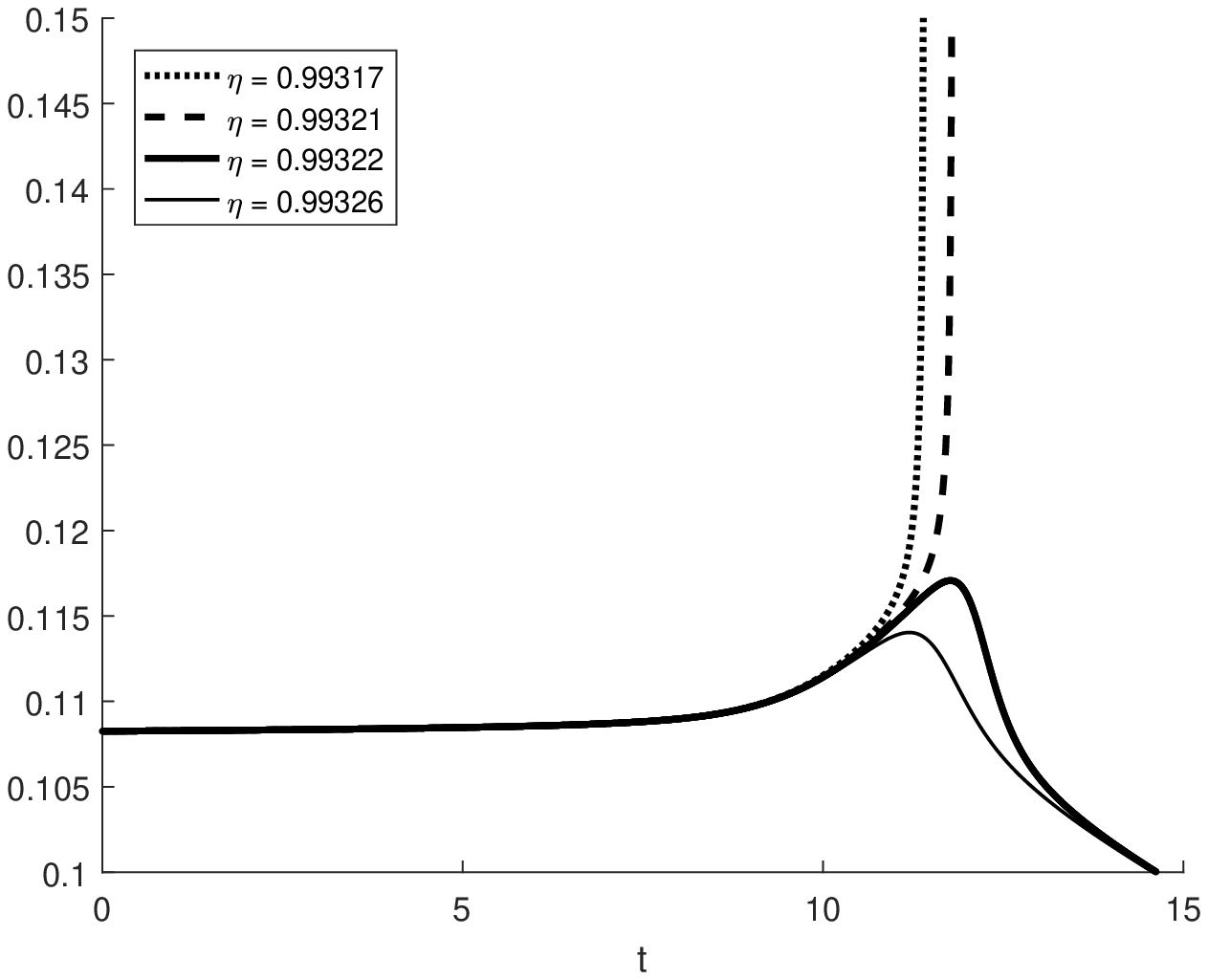}
}
\subfloat{
\includegraphics[width=0.49\textwidth]{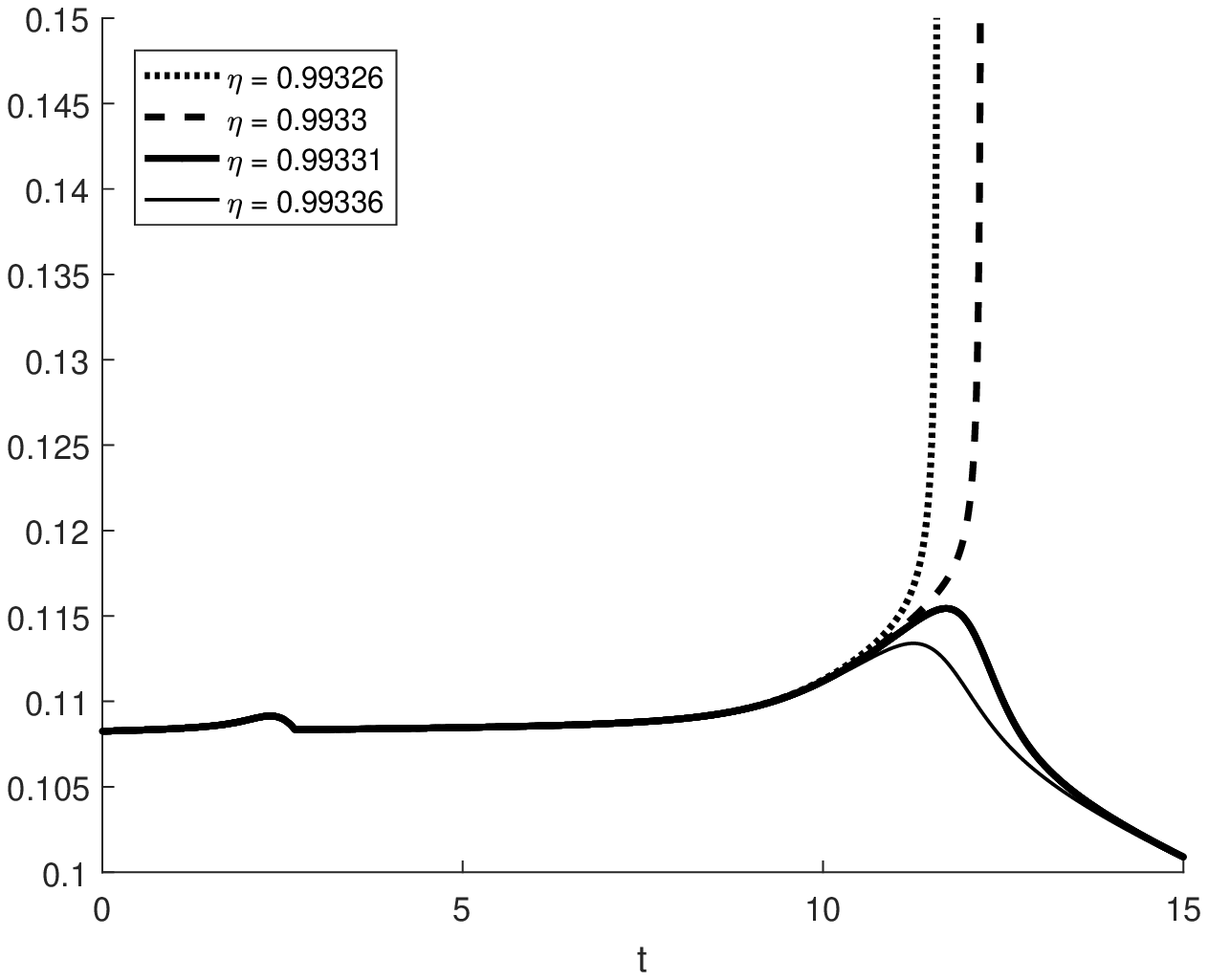}
}
\caption{Running supremum for initial value 1 (left) and 2 (right)}
\label{fig:ResMacRunSup}
\end{figure}

\begin{figure}[htb!]
\centering
\subfloat{
\includegraphics[width=0.49\textwidth]{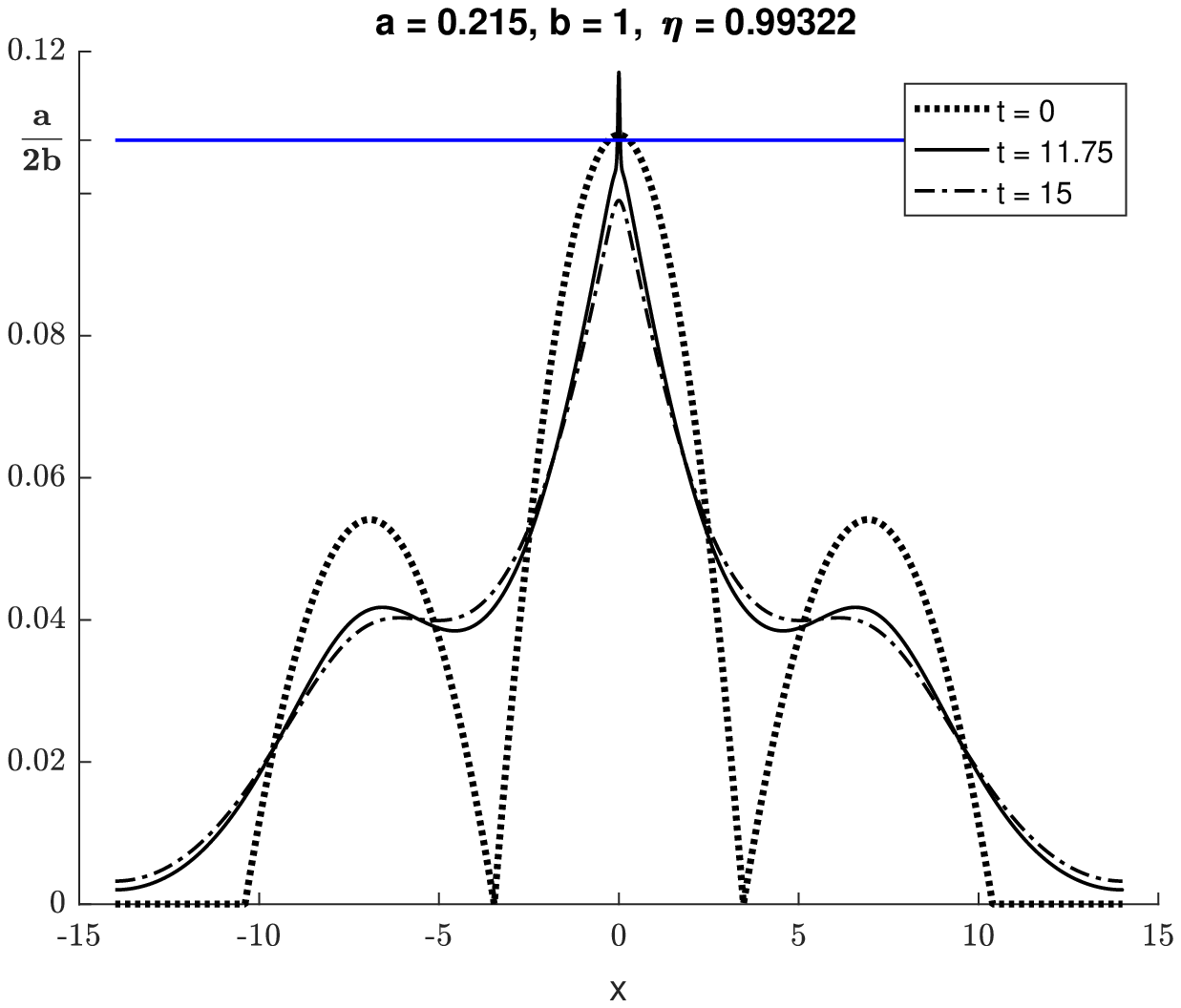}
}
\subfloat{
\includegraphics[width=0.49\textwidth]{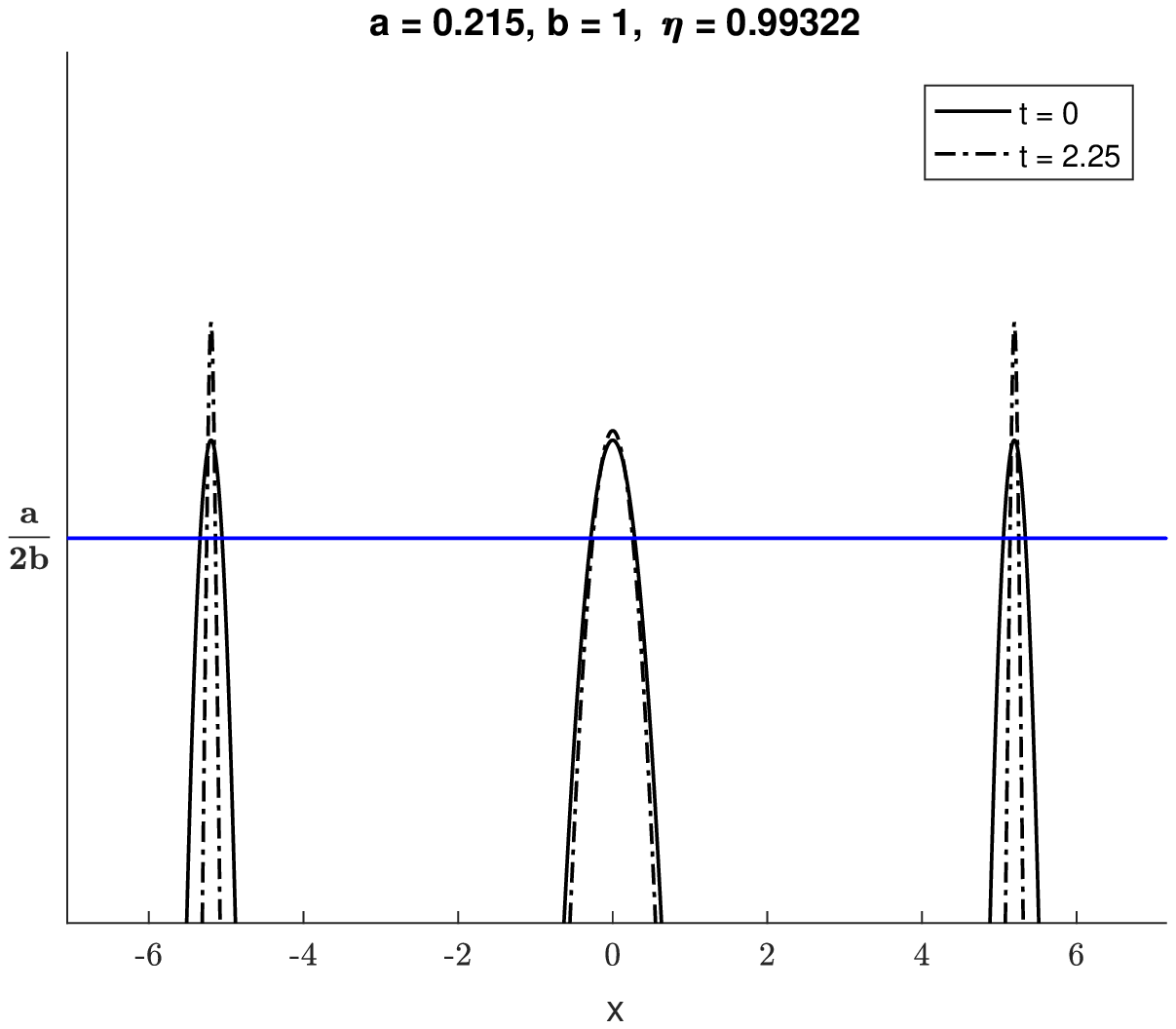}
}
\caption{Aggregation and then diffusion for initial data 1 (left); zoom-in at the small aggregation in the case of initial data 2 (right)}
\label{fig:ResMac}
\end{figure}

In figure \ref{fig:ResMacAggregation}, the simulation results for the aggregation regime are shown. 
However, once the values are above $\frac{a}{2b}$, the numerical approximation starts to peak and blows up, i.e. the numerical solution collapses completely.

\begin{figure}[htb!]
\centering
\subfloat{
\includegraphics[width=0.49\textwidth]{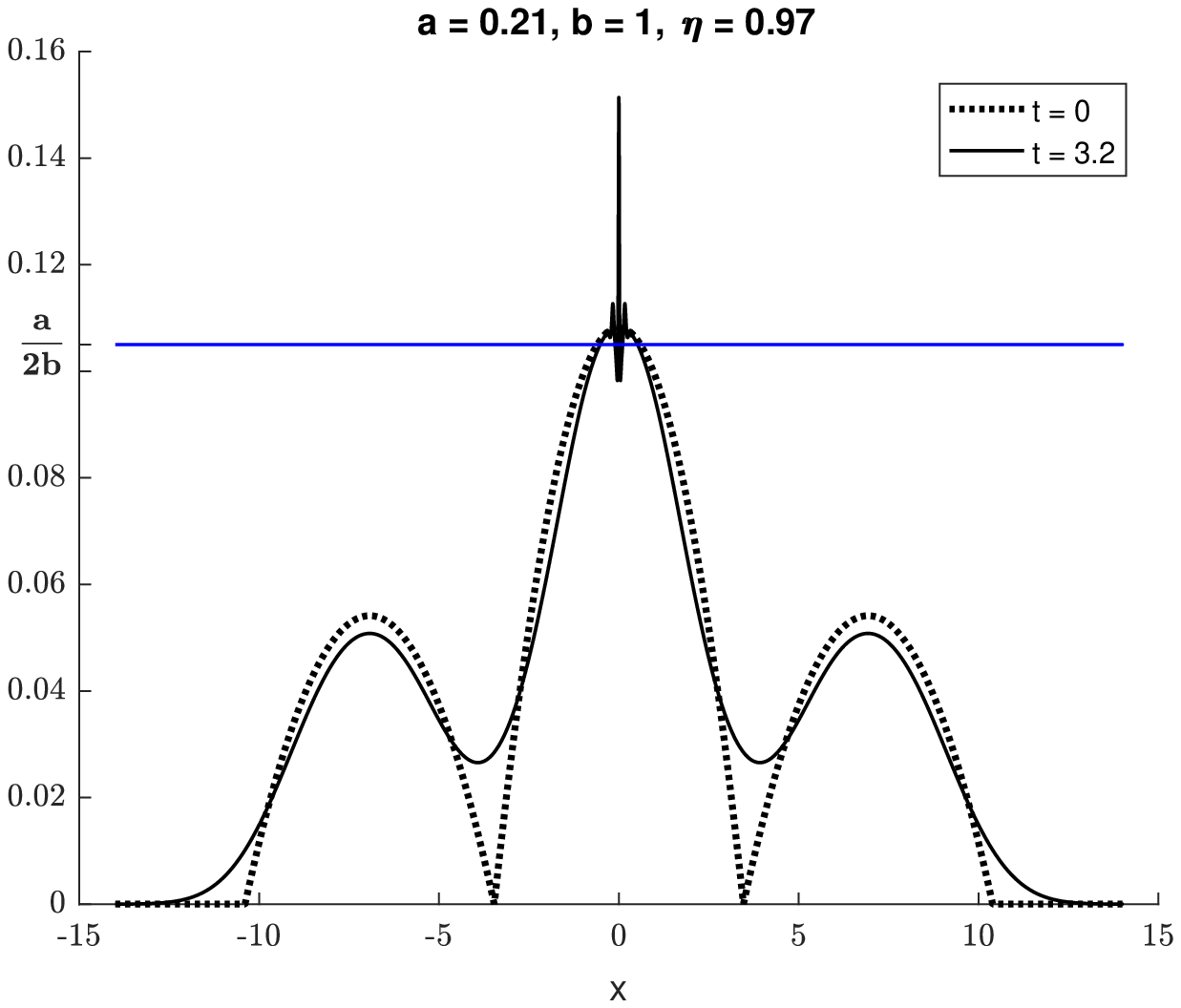}
}
\subfloat{
\includegraphics[width=0.49\textwidth]{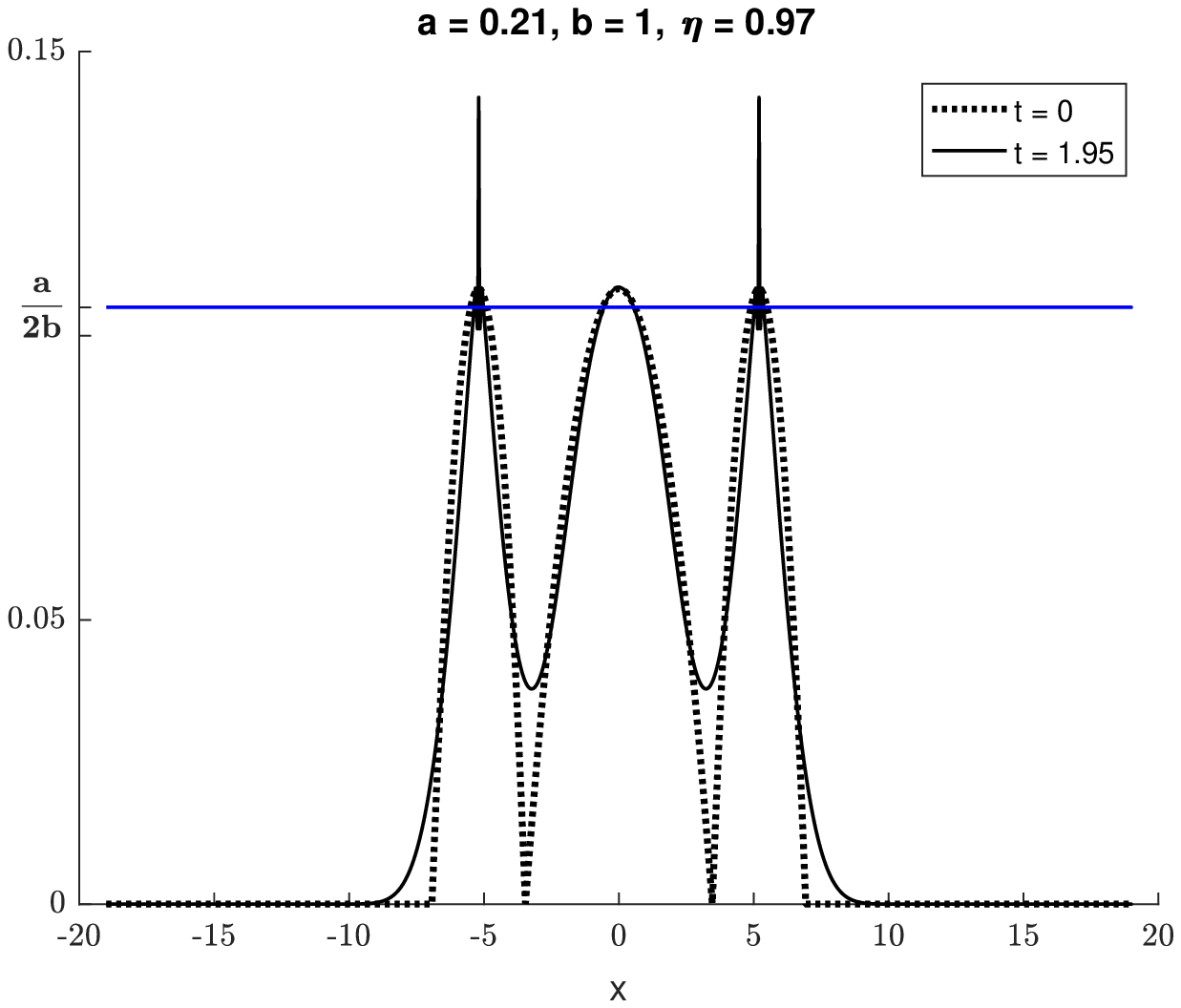}
}
\caption{Aggregation for $\eta = 0.97$ and initial data 1 (left) and 2 (right)}
\label{fig:ResMacAggregation}
\end{figure}

\subsection{Convergence of the stochastic particle system to the diffusion-\\aggregation equation}
In the previous part, we have analyzed the performance of numerical solutions separately.  
Since in the diffusive regime, i.e. $\eta>1$, the density of the particle system is expected to converge to the 
density of the diffusion-aggregation model, we now study the convergence numerically. 

Let $\hat{X}_j^{i,m}$ be the $m-$th sample of $X_{t_j}^i$ for $m=1,\dots,M$. We define 
\begin{align*}
u_i^j = \frac{1}{M} \sum_{m=1}^M \frac{1}{N \Delta x}\sum_{k=1}^N \chi_{[x_i-\Delta x/2,x_i+\Delta x/2)}(\hat{X}_j^{k,m})
\end{align*}
as the density estimator for the particle system.
Let $\tilde{u}^j_i$ denote the macroscopic density approximation on the same time-space grid.  
We define the error by $e_i^j:= u_i^j-\tilde{u}_i^j$ and use the following norms
\begin{align*}
||e||_\infty = \max_{i,j}\{|e_i^j|\}, \quad
||e||_{p} = \max_{j}\left(\Delta x \sum_{i} |e_i^j|^p\right)^{\frac{1}{p}}
\end{align*}
to measure the distance between both approximations. 
We consider the diffusion case $\eta = 1.5$ and study the convergence of the estimated particle to the macroscopic density 
regarding the number of particles $N$. 
Even for the rough spatial discretization $\Delta x = 2^{-3}$ and $1000$ Monte-Carlo runs, we observe a convergence 
in all norms as table \ref{tab:conv1} shows. 
\begin{table}[H]
\centering
\begin{tabular}{c|c|c|c|c|c|c}
&$||\cdot||_\infty$&EOC&$||\cdot||_{1}$&EOC&$||\cdot||_{2}$&EOC\\ \hline
$N = 50$&$2.199 \cdot 10^{-2}$&$-$&$7.518 \cdot 10^{-2}$&$-$&$2.026 \cdot 10^{-2}$&$-$\\ \hline
$N = 100$&$1.508 \cdot 10^{-2}$&$0.544$&$6.365 \cdot 10^{-2}$&$0.240$&$1.725 \cdot 10^{-2}$&$0.232$ \\ \hline
$N = 200$&$1.383 \cdot 10^{-2}$&$0.130$&$5.691 \cdot 10^{-2}$&$0.161$&$1.539 \cdot 10^{-2}$&$0.165$ \\ \hline
$N = 400$&$1.150 \cdot 10^{-2}$&$0.266$&$5.114 \cdot 10^{-2}$&$0.154$&$1.413 \cdot 10^{-2}$&$0.123$ \\ \hline
$N = 800$&$1.111 \cdot 10^{-2}$&$0.050$&$5.057 \cdot 10^{-2}$&$0.016$&$1.381 \cdot 10^{-2}$&$0.033$ \\ \hline
Mean EOOC&$-$&$0.248$&$-$&$0.143$&$-$&$0.138$
\end{tabular}
\caption{Numerical convergence in $N$ with respect to different norms for initial distribution 1, $\eta = 1.5, \Delta x = 2^{-3}$ and time horizon $T=7$ }
\label{tab:conv1}
\end{table}
The EOC decreases as the number of particles increases which is the result of the rough spatial discretization and the high value 
of $\epsilon$. We note that this gap cannot be reduced by a higher number of particles. 
If the range of strong interaction $\epsilon$ and the spatial discretization is reduced, 
we would need a very large number of particles (see \eqref{eq:Estimate}) as well as a small time step-size to obtain meaningful results
since the computation time increases at least quadratically in the number of particles.

\section*{Acknowledgments}
This work was financially supported by the DAAD project ``DAAD-PPP VR China'' (project ID: 57215936)
and the DFG grant GO 1920/4-1.
\bibliographystyle{siam}
\bibliography{ms}
\end{document}